\documentclass[article,ij4uq]{ij4uq}              

\usepackage[hang]{footmisc}
\usepackage[ruled]{algorithm2e}
\usepackage{algorithmic}
\usepackage{subfigure}

\usepackage{amsfonts}
\usepackage{amsmath}
\usepackage{amssymb}
\usepackage{graphicx}
\usepackage{mathrsfs}
\usepackage{epsfig}
\usepackage{color}
\usepackage[normalem]{ulem}
\usepackage{bm}
\usepackage{pifont}

\def\bE{{\mathbb{E}}}
\def\cJ{{\mathcal{J}}}
\def\cR{{\mathcal{R}}}
\def\cU{{\mathcal{U}}}
\def\cD{{\mathcal{D}}}

\def\Hep{{\mathrm{H}}} 
\def\hep{{\mathrm{h}}} 

\def\bx{{\boldsymbol{x}}}
\def\by{{\boldsymbol{y}}}

\DeclareMathOperator*{\var}{Var}

\def\mfku{{\mathfrak{u}}}

\newcommand{\dotW}{\dot{W}}

\newcommand{\ibxi}{\boldsymbol{\xi}}

\newcommand{\bbF}{{\mathbb{F}}}
\newcommand{\bbR}{\mathbb{R}}
\newcommand{\bbN}{\mathbb{N}}
\def\cF{{\mathcal{F}}}
\newcommand{\bF}{\boldsymbol{F}}
\newcommand{\bc}{\boldsymbol{c}}
\newcommand{\bb}{\boldsymbol{b}}
\newcommand{\bu}{\mathbf{u}}
\newcommand{\tu}{\tilde{u}}
\newcommand{\wprod}{\diamond}
\newcommand{\sI}{\mathsf{I} }
\newcommand{\sII}{\mathsf{II} }
\newcommand{\tmc}{\textrm{mc}}
\newcommand{\iba}{\boldsymbol{\alpha}}
\newcommand{\ibb}{\boldsymbol{\beta}}
\newcommand{\ibg}{\boldsymbol{\gamma}}

\newcommand{\ibk}{\boldsymbol{\kappa}}

\newcommand{\Wphi}{{\sigma G}}
\newcommand{\phivar}{{\sigma^2}}

\def\revdraft{0}
\if\revdraft 1

\newcommand{\Yu}[1]{{\color{blue}#1}}

\newcommand{\WanRev}[2]{{\color{blue}{\sout{#1}}#2}}

\else

\newcommand{\Yu}[1]{{#1}}

\newcommand{\WanRev}[2]{{#2}}
\fi

\newcommand{\remove}[1]{}

\newtheorem{lemma}[theorem]{Lemma}

\fancypagestyle{plain}{%
	\fancyhf{}
	\fancyhead[R]{\small {\it International Journal for Uncertainty Quantification}, x(x): \thepage--\pageref{LastPage} (\myyear\today)}
	\fancyfoot[R]{\small\bf\thepage }
	\fancyfoot[L]{\fottitle}
}
\renewcommand{\myyear}{2019}
\renewcommand{\today}{}

\begin{document}
	
	\volume{Volume x, Issue x, \myyear\today}
	\title{Numerical approximation of elliptic problems with log-normal 
		random coefficients}
	\titlehead{Numerical approximation of elliptic problems with log-normal 
		random coefficients}
	\authorhead{X. Wan \& H. Yu}
	\author[1]{Xiaoliang Wan}
	\author[2, 3]{Haijun Yu}
	\address[1]{Department of Mathematics\\
      Center for Computation and Technology\\
      Louisiana State University\\
      Baton Rouge, LA 70803\\
      Email: {xlwan@math.lsu.edu}}

    \address[2]{NCMIS \& LSEC, Institute of Computational
      Mathematics and Scientific/Engineering Computing,
      Academy of Mathematics and Systems Science, Beijing
      100190,
      China\\
      Email: {hyu@lsec.cc.ac.cn}}
	
	\address[3]{School of Mathematical Sciences, University
      of Chinese Academy of Sciences\\
      Beijing 100049, China}
	
	\dataO{mm/dd/yyyy}
	\dataF{mm/dd/yyyy}

\abstract{ In this work, we consider a non-standard
      preconditioning strategy for the numerical
      approximation of the classical elliptic equations with
      log-normal random coefficients. In \cite{Wan_model}, a
      Wick-type elliptic model was proposed by modeling the
      random flux through the Wick product. Due to the
      lower-triangular structure of the uncertainty
      propagator, this model can be approximated efficiently
      using the Wiener chaos expansion in the probability
      space. Such a Wick-type model provides, in general, a
      second-order approximation of the classical one in
      terms of the standard deviation of the underlying
      Gaussian process. Furthermore, when the correlation
      length of the underlying Gaussian process goes to
      infinity, the Wick-type model yields the same solution
      as the classical one. These observations imply that
      the Wick-type elliptic equation can provide an
      effective preconditioner for the classical random
      elliptic equation under appropriate conditions. We use
      the Wick-type elliptic model to accelerate the Monte
      Carlo method and the stochastic Galerkin finite
      element method. Numerical results are presented and
      discussed.  }


\keywords{Wiener chaos expansion; Wick product;
      Stochastic elliptic PDE; Uncertainty quantification;
      log-normal random coefficient}

\maketitle

\section{Introduction}
    
Numerical approximation of elliptic problems with log-normal
random coefficients has \WanRev{recently}{}  received a lot of
attention. We consider the following mathematical model
\begin{equation}\label{eqn:ellip_omega}
  \textrm{Model I:}\quad
  \left\{
    \begin{array}{rcll}
      -\nabla\cdot(a(\bx,\omega)\nabla u(\bx,\omega))&=&f(\bx),&\bx\in D,\\
      u(\bx,\omega)&=&0,&\bx\in\partial D,
    \end{array}
  \right.
\end{equation}
where $\ln a(\bx,\omega)$ is a second-order homogeneous
Gaussian random process, and the force term is assumed to be
deterministic for simplicity.  We call problem
\eqref{eqn:ellip_omega} model I in this paper.  Theoretical
difficulties of problem \eqref{eqn:ellip_omega} are mainly
related to the lack of uniform ellipticity, where the
Lax-Milgram lemma is not applicable. The existence and
uniqueness of the solution of problem
\eqref{eqn:ellip_omega} are usually established with respect
to a weighted norm \cite{LRW09,Sarkis09,Gittelson_MMMAS10}
or a weighted measure \cite{Starkloff11}, or by using the
Fernique theorem
\cite{Prato92,Charrier_SINUM12}. Considering the Wiener
chaos approach and Galerkin projection
\cite{Ghanem91,LRW09}, the difficulties of numerical
approximation of problem \eqref{eqn:ellip_omega} are
twofold: First, if we start from the theoretical study
\cite{Sarkis09,Starkloff11}, a different test space rather
than $L_2(\bbF;H_0^1(D))$ is required,,which may be not easy
to construct. \Yu{Here $\bbF:=\left(\Omega,\cF,P\right) $ is
  the probability space for $\omega$, detailed presentation
  of $\bbF$ is given in Section 2.} Second, if we choose
$L_2(\bbF;H_0^1(D))$ as the test space and use Wiener chaos
as the basis for the probability space, although no
divergence with respect to $L_2(\bbF;H_0^1(D))$ norm has
been numerically observed (the solution of problem
\eqref{eqn:ellip_omega} actually belongs to
$L_2(\bbF;H_0^1(D))$ \cite{Charrier_SINUM12}), the stiffness
matrix is full and dense.  In other words, an efficient
preconditioner is required. Study of elliptic problems with
other types of random coefficients can be found in
\cite{Babuska_SIAMJNA04,Schwab_CMAME05,Schwab_IMANUM07},
etc.

\Yu{ The elliptic equation with log-normal random
  coefficient has been studied by means of the perturbation
  technique (see, e.g., \cite{BN_SIAMJUQ14},
  \cite{BNK_CMAME16}), which has been also employed for
  other types of random coefficients (see, e.g.,
  \cite{CCDS_ESAIM13}). However, the perturbation method
  only works for small variability of random coefficient and
  low degree of the Taylor polynomial \cite{BN_SIAMJUQ14}.
 
  Another approach is to construct an auxiliary problem as
  some sort of preconditioner of the original problem,
  e.g. the idea of using a smoother version of the original
  problem (generated by a smoothing kernel) in a Monte Carlo
  control variate approach has been discussed by Nobile
  et. al.  \cite{NT_SPDE_2015}, \cite{NTTT_SGA2016}. Other
  known preconditioning skills include the traditional
  algebraic preconditioner \cite{PE_IMAJNA_2009},
  \cite{PU_SIAMJMAA2010} and the bi-fidelity method
  \cite{HFND_2017}.  }

\Yu{In this paper we take a new approach to construct an
  auxiliary problem used as a preconditioner of model I.} From
the modeling point of view, the randomness can be introduced
in different ways.  A typical strategy is to replace the
flux $a\nabla u$ as $a\wprod\nabla u$ with $\wprod$ being
the Wick product \cite{Holden96,Theting_SSR2000,Wan_PNAS09},
motivated by the observations that the Wick product is
consistent with Skorohod stochastic integral in a Hilbert
space and can smooth the irregularity induced by white
noise. Once the Wick product is adopted, the equations for
the coefficients of Wiener chaos expansion are decoupled and
can be solved one-by-one. Although this is a very nice
property for numerical computation, the original equation is
changed and the model difference becomes the main
concern. In \cite{Wan_model,Wan_model2}, a new Wick-type
model was proposed by modeling the flux as
$\left(a^{-1}\right)^{\wprod(-1)}\wprod\nabla u$:
\begin{equation}\label{eqn:ellip_Wick}
  \textrm{Model II:}
  \quad \left\{
    \begin{array}{rcll}
      -\nabla\cdot\left(\left(a^{-1}\right)^{\wprod(-1)}(\bx,\omega)\wprod\nabla 
      u(\bx,\omega)\right)&=&f(\bx),&\bx\in D,\\
      u(\bx,\omega)&=&0,&\bx\in\partial D,
    \end{array}\right.
\end{equation}
which we call model II in this paper.  In general, both
fluxes $a\wprod\nabla u$ and
$\left(a^{-1}\right)^{\wprod(-1)}\wprod\nabla u$ will
introduce a second order approximation of the solution of
model I in terms of the standard deviation ($\sigma<1$) of
the underlying Gaussian process. However, the latter choice
provides a much smaller difference.  Actually when the
correlation length of the underlying Gaussian process goes
to infinity, model II has the same solution as model I. In
addition, the uncertainty propagator of model II is also
lower-triangular, which can be solved efficiently. Another
way to approximate the flux $a\nabla u$ using the Wick
product is to employ the Mikulevicius-Rozovskii (M-R)
formula \cite{Boris_SNS}, which shows that the product of
two random variables, say $X$ and $Y$, has a Taylor-like
expansion
\begin{equation}
  XY=X\wprod Y+\sum_{n=1}^\infty\frac{\cD^n X\wprod\cD^n Y}{n!},
\end{equation} 
where $\cD$ indicates the Malliavin derivative
\cite{Nualart06}. It is seen that $X\wprod Y$ is the
lowest-order term in this expansion. We can include more
terms from the M-R formula to get a better approximation of
$a\nabla u$ \cite{Venturi_2013,Wan_model3}.  It is shown in
\cite{Wan_model3} that with respect to the truncation order
$Q$ of the Malliavin derivative and the standard deviation
of the underlying Gaussian process such a strategy provides
a difference of $\mathcal{O}(\sigma^{2(Q+1)})$ from the
solution of model I.  However, upon doing so, the
corresponding uncertainty propagator will be not
lower-triangular any more, although the coupling in the
upper-triangular part will be weak if the truncation order
in the M-R formula is relatively small.

In this work, we will explore the possibility to use model
II as a predictor to improve some algorithms for model I
since model II can be approximated efficiently and the
difference between models I and II can be very
small. Depending on the properties of the random
coefficient, we mainly consider the Monte Carlo method and
the Wiener chaos approach with Galerkin projection for model
I.

This paper is organized as follows. In section
\ref{sec:preliminary}, we define the Wiener chaos space and
the Wick product. Stochastic elliptic models are discussed
in section \ref{sec:models} and the corresponding
uncertainty propagators are given in section
\ref{sec:galerkin}. Numerical algorithms are proposed in
section \ref{sec:algorithm}. We present numerical results in
section \ref{sec:results}, followed by a summary section.

\section{Wiener chaos space and Wick product}
\label{sec:preliminary}
\Yu{ Since the underlying random {variables} of the model are
 {i.i.d.} Gaussian, whose corresponding stochastic orthogonal
  polynomials are Hermite. We first introduce basic
  properties of Hermite polynomials.

\subsection{Hermite polynomials}
 
The one-dimensional (probabilistic) Hermite polynomials
of degree $n$ are defined as
	\begin{equation}
	\Hep_n(\xi) := (-1)^n e^{\frac{\xi^2}{2}} \frac{d^n}{d\xi^n} e^{-\frac{\xi^2}{2}}.
	\end{equation}
	$\Hep_n(\xi)$ are orthogonal
	with respect to the weight
	$\frac{1}{\sqrt{2\pi}}e^{-\frac{\xi^2}{2}}$, in the sense
	\begin{equation}
	\int_{-\infty}^\infty \Hep_m(\xi)\Hep_n(\xi) 
	\frac{1}{\sqrt{2\pi}} e^{-\frac{\xi^2}{2}} d\xi
	= n! \delta_{nm}.
	\end{equation}
The values of Hermite polynomials can be evaluated using
the following three-term recurrence formula:
\begin{align*}
& \Hep_0(\xi)=1,\qquad \Hep_1(\xi)=\xi,\\
& \Hep_{n+1}(\xi)= \xi \Hep_n(\xi) - n \Hep_{n-1}(\xi),\quad n\ge 2.
\end{align*}
Hermite polynomials satisfy a very simple derivative
relation:
\begin{equation}
\Hep_n'(\xi)= n \Hep_{n-1}(\xi)\quad \forall\, n\ge 0.
\end{equation}
We list below in Lemma \ref{lem:He_prop} several properties
of Hermite polynomials, which will be used later.
\begin{lemma}\label{lem:He_prop}
  For one-dimensional Hermite polynomials, the following
  properties hold
  \begin{align}
    &
      \exp\left(s\xi-\frac{1}{2}s^2\right)
      = \sum_{i=0}^\infty\frac{s^i}{i!}\Hep_i(\xi),
      \label{eqn:genfunc} \\ 
    &
      \Hep_n(\xi+s)
      = \sum_{i=0}^n\binom{n}{i}s^{n-i}\Hep_{i}(\xi),
      \label{eqn:shift_hermite}\\
    &
      \Hep_i(\xi)\Hep_j(\xi)
      = \sum_{k\leq i\wedge j}\chi(i,j,k)\Hep_{i+j-2k}(\xi).
      \label{eqn:hermite_multiply_expansion}
  \end{align}
  where $s\in\bbR$, \Yu{$i\wedge j := \min\{i,j\}$} and
  \[
    \chi(i,j,k)=\frac{i!j!}{k!(i-k)!(j-k)!}.
  \]
\end{lemma}

\subsection{Wick product}

Now we list the definition and some basic properties of Wick
product, which can be found in existing literature
(e.g. \cite{Holden96}, \cite{HuYan2009}).

The Wick product of a set of random variables with finite
moments is defined recursively as follows:
\[
\left\langle \emptyset\right\rangle =1,\qquad
\frac{\partial\left\langle X_{1},\ldots,X_{k}\right\rangle
}{\partial X_{i}}=\left\langle X_{1},\ldots,X_{i-1},
X_{i+1},\ldots,X_{k}\right\rangle, \quad k\ge 1, 
\]
together with the constraint that the average is zero
\[
\bE\left\langle X_{1},\ldots,X_{k}\right\rangle =0,\quad k\ge 1.
\]
It follows that
\[
\langle X\rangle =X-\bE[X],\quad \langle
X,Y\rangle =XY-\bE[Y]X-\bE[X]Y+2\bE[X]\bE[Y]-\bE[XY].
\]
If $X,Y$ are independent, from about formula, we know
\[
\left\langle X,Y\right\rangle =\left\langle X\right\rangle
\left\langle Y\right\rangle .
\]
On the other hand, if $Y=X$, we get
\[
\left\langle X,X\right\rangle
=X^{2}-2\bE[X]X+2\bE[X]^{2}-\bE[X^{2}].
\]
Define $X\wprod Y:= \langle X, Y \rangle$ and
\[
P_{n}(X) := X^{\wprod n} =\underbrace{\left\langle
	X,\ldots,X\right\rangle }_{n\text{ times}},
\]
then $P_{n}'(x)=nP_{n-1}(x)$. 

Wick product is closely related to Hermite polynomials. If
$\xi$ is a normally distributed variable with variance $1$,
then
\begin{equation}\label{eq:Wick_powerHermite}
\xi^{\wprod n}=\Hep_{n}(\xi).
\end{equation}
and
\begin{equation}\label{eq:Wick_Hermite}
\Hep_{n}(\xi)\wprod \Hep_{m}(\xi)=\Hep_{n+m}(\xi).
\end{equation}

Using Taylor series, one can define the exponential function
of Wick product as
\begin{equation}\label{eq:wick_exponential}
e^{\wprod X}:=\sum_{n=0}^{\infty}\frac{1}{n!}X^{\wprod n}.
\end{equation}
{For a} normally distributed variable $\xi$, it can be checked
that \cite{Holden96}
\begin{equation}\label{eqn:exp_diamond}
e^{\wprod\left[\sigma\xi\right]}=e^{\sigma\xi-
\sigma^2/2},
\end{equation}
\begin{equation}\label{eqn:exp_dia_inverse}
e^{\wprod\left[\sigma\xi\right]}\wprod e^{\wprod \left[-\sigma\xi\right]}
=1,
\end{equation} 
and the following statistics hold
\begin{equation}
\bE\left[e^{\wprod\left[\sigma\xi\right]}\right]=1,\quad
\var\left[e^{\wprod\left[\sigma\xi\right]}\right]=
e^{\sigma^2}-1.
\end{equation}
}

\subsection{Wiener chaos space}

We define $\bbF:=\left(\Omega,\cF,P\right) $
as a complete probability space, where $\cF$ is the
$\sigma$-algebra generated by the countably many
i.i.d. Gaussian random variables
$\left\{\xi_{k}\right\}_{k\geq1}$. Define
$\ibxi:=(\xi_1,\xi_2,\ldots)$.  Let $\cJ$ be the collection
of multi-indices $\iba$ with
$\iba=(\alpha_{1},\alpha_{2}, \ldots)$ so that
$\alpha_{k}\in\bbN_0$ and
$|\iba|:=\sum_{k\geq1}\alpha_{k}<\infty$. For
$\iba,\ibb\in\cJ$, we define
\[
\iba+\ibb=(\alpha_{1}+\beta_{1},\alpha_{2}+\beta_{2},\ldots),\quad
\iba!=\prod_{k\geq 1}\alpha_{k}!,\quad
\binom{\iba} {\ibb}=\prod_{k\geq
	1}\binom{\alpha_k}{\beta_k}.
\]
We use $(\Yu{\bm 0})$ to denote the multi-index with
all zero entries: $(\Yu{\bm{0}})_k=0$ for all
$k$. Define the collection of random variables
\Yu{ $\Xi$ as follows:
	\begin{equation}\label{eq:gPCset}
	\Xi := \{\hep_{\iba },\iba\in\cJ\},
	\quad
	\hep_{\iba}(\ibxi) := \prod_{k\geq1}
	\frac{1}{\sqrt{\alpha_{k}!}} \Hep_{{\alpha}_{k}}(\xi_{k}),
	\end{equation} 
  } where $\Hep_{n}(\xi)$ are the one-dimensional
  (probabilistic) Hermite polynomials.  \Yu{ For convenience,
    we also define
\begin{equation}
  \Hep_{\iba}(\ibxi) := \prod_{k\geq1}
  \Hep_{{\alpha}_{k}}(\xi_{k}).
\end{equation}
For any fixed $k$-dimensional i.i.d. Gaussian random
variable $\bm \xi$, the following relations hold
\begin{equation}
  \bE[\Hep_{\iba}(\ibxi)\Hep_{\ibb}(\ibxi)]=
  \delta_{\iba\ibb}\iba!,\quad
  \bE[\hep_{\iba}(\ibxi)\hep_{\ibb}(\ibxi)]=
  \delta_{\iba\ibb}.
\end{equation}
The set $\Xi$ {forms} an orthonormal basis for $L_{2}(\bbF)$
\cite{Cameron_AM47}, that is: if $\eta\in L_{2}(\bbF)$, then
\begin{equation}
  \eta=\sum_{\iba\in\cJ}\eta_{\iba} \hep_{\iba},\quad
  \eta_{\iba}=\bE[\eta\hep_{\iba}]
\end{equation}
and
\begin{equation}
  \bE[\eta^{2}]=\sum_{\iba\in\cJ}\eta_{\iba}^{2}.
\end{equation}
The Wick product of multi-dimensional stochastic Hermite
polynomials are:}
\begin{equation}\label{eq:Hermite_Wickprod}
  \Hep_{\iba}(\ibxi)\wprod\Hep_{\ibb}(\ibxi)
  =\Hep_{\iba+\ibb}(\ibxi),
  \qquad
  \hep_{\iba}(\ibxi)\wprod\hep_{\ibb}(\ibxi)
  =\sqrt{\frac{(\iba+\ibb)!}{\iba!\ibb!}}\hep_{\iba+\ibb}(\ibxi).
\end{equation}	

\Yu{Note that if we consider the expansion of
  $\Hep_{\iba}(\ibxi)\Hep_{\ibb}(\ibxi)$ using the base set
  $\Xi$}, it is obvious that there exist low-order terms
\Yu{in addition to} $\Hep_{\iba+\ibb}(\ibxi)$; however, in
the definition of Wick product, all these low-order terms
are removed, \Yu{cf. equation
  \eqref{eqn:hermite_multiply_expansion} and equation
  \eqref{eq:Hermite_Wickprod}}. Such a difference of the
Wick product from the regular multiplication stems from the
fact that the Wick product should be interpreted from the
viewpoint of stochastic integral. The correspondence between
the Wick product and the Ito-Skorokhod integral can be found
in \cite{Holden96,Nualart06, Rozovskii_SIMA09,Wan_PNAS09}.

For the numerical approximation, the number of Gaussian
random variables and the polynomial order need to be
truncated. We define
\begin{equation}
  \cJ_{M,p}=\{\iba|\iba=(\alpha_1,\ldots,\alpha_M),\,|\iba|\leq p\},
\end{equation}
where $p\in\bbN_0$ is the maximum total degree. (To reduce
the number of stochastic bases, one can also consider
\Yu{the sparse {grids} or sparse spectral Galerkin method , see
  e.g. \cite{ShenY10,ShenY12, CCDS_ESAIM13,SWY_JCAM2014,
    NTTT_SGA2016}}, \WanRev{}{where} the overall procedure is similar.)
Correspondingly, $\ibxi$ is split into two parts
\[
  \ibxi=\ibxi_1\oplus\ibxi_2=(\xi_1,\ldots,\xi_M)\oplus(\xi_{M+1},\ldots).
\]
For simplicity, we use $\ibxi$ for both finite-dimensional
and infinite-dimensional cases, and the dimensionality will
be indicated by the set $\cJ$ or $\cJ_{M,p}$ for the
index. Let $N_{M,p}$ be the cardinality of $\cJ_{M,p}$. It
is obvious that there exists a one-to-one correspondence
between $1\leq i\leq N_{M,p}$ and $\iba\in\cJ_{M,p}$. We use
$i(\iba)$ or $\iba(i)$ to indicate such a one-to-one mapping
whenever necessary.

Given a real separable Hilbert space $X$, we denote by
$L_{2}({\bbF};X)$ the Hilbert space of square-integrable
${\cF}$-measurable $X$-valued random elements $f$. When
$X={\bbR}$, we write $L_{2}({\bbF})$ instead of
$L_{2}({\bbF};{\bbR})$. Given a collection
$\cR=\{r_{\iba},\ \iba\in \cJ\}$ of positive real numbers
\Yu{with an upper bound $R$, i.e.}
$r_{\iba}<R$ for all $\iba$, we define the space
$\cR L_2(\bbF;X)$ as the closure of $L_2(\bbF;X)$ in the
norm
\begin{equation}\label{norm-w}
  \|u\|_{\cR L_2(\bbF;X)}^2=
  \sum_{\iba\in \cJ} r_{\iba}\|u_{\iba}\|_X^2,
\end{equation}
where $u=\sum_{\iba\in\cJ}u_{\iba}\hep_{\iba}(\ibxi)$.  The
space $\cR L_2(\bbF;X)$ is called a weighted chaos
space\Yu{, it is a natural norm for the stochastic space
  using Karhunen-Lo\'eve expansion}.  In this work, $X$ is
chosen as $H_0^1(D)$ for elliptic problems with homogeneous
boundary conditions.

\section{Stochastic elliptic models}\label{sec:models}

In this paper, we consider the following two stochastic
elliptic models:
\begin{subequations}
  \begin{align}
    \textrm{Model I: }
    &
      -\nabla\cdot(a(\bx,\omega)\nabla u_{\sI}(\bx,\omega))
      =f(\bx),\label{eqn:u_I}\\
    \textrm{Model II: }
    &
      -\nabla\cdot\left(\left(a^{-1}\right)^{\wprod(-1)}
      (\bx,\omega)\wprod \nabla u_{\sII}(\bx,\omega)\right)
      =f(\bx),\label{eqn:u_III}
  \end{align}
\end{subequations}
with boundary condition $u(\bx, \omega)=0$ on $\partial D$,
where
$a^{-1}(\bx,\omega)\wprod\left(a^{-1}(\bx,\omega)\right)^{\wprod(-1)}=1$. In
particular, we assume that the force term $f(\bx)$ is
deterministic for simplicity and the random coefficient
$a(\bx,\omega)$ takes the following form
\begin{equation}\label{eqn:lognormal}
  a(\bx,\omega)=e^{\wprod (\sigma G(\bx,\omega))}=e^{
    \sigma G(\bx,\omega)-\frac{1}{2}\sigma^2},
\end{equation}
where $G(\bx,\omega)$ is a stationary Gaussian random
process with zero mean and unit variance, subject to a
normalized covariance kernel
$K(\bx_1,\bx_2)=K(|\bx_1-\bx_2|)=\bE[G(\bx_1,\omega)G(\bx_2,\omega)]$.
\remove{ More details about definitions of the exponential
  function with respect to Wick product, and properties of
  the chosen log-normal random process can be found in
  \ref{sec:smoothed_white_noise}.  } According to the Mercer
theorem \cite{Riesz90}, $K(\bx_1,\bx_2)$ has an expansion as
\begin{equation}
  K(\bx_1,\bx_2)=\sum_{i=1}^\infty \lambda_i\phi_i(\bx_1)\phi_i(\bx_2),
\end{equation}
where $\{\lambda_i,\phi_i(\bx)\}_{i=1}^\infty$ are
eigen-pairs of $K(\bx_1,\bx_2)$ satisfying
\begin{equation}
  \int_DK(\bx_1,\bx_2)\phi_i(\bx_2)d\bx_2=\lambda_i\phi_i(\bx_1),\quad
  \int_D\phi_i(\bx)\phi_j(\bx)d\bx=\delta_{ij}.
\end{equation}
Then $G(\bx,\omega)$ has the following Karhunen-Lo\`{e}ve
\Yu{(K-L)} expansion
\begin{equation}\label{eqn:KL}
  G(\bx,\omega)=\sum_{i=1}^\infty\sqrt{\lambda_i}\phi_i(\bx)\xi_\Yu{i},
\end{equation}   
where $\xi_k$ are independent Gaussian random
variables. Furthermore,
\begin{equation}\label{eqn:var}
  \sum_{i=1}^\infty\lambda_i\phi^2_i(\bx)
  = K(\bx,\bx)=\bE[G^2(\bx,\omega)]=1,
  \quad \forall \bx\in D.
\end{equation}

Using equations \eqref{eqn:KL}, \eqref{eqn:var} and
\eqref{eqn:genfunc}, we can obtain the Wiener chaos
expansion of the log-normal random process $a(\bx,\omega)$
\begin{equation}\label{eqn:a_chaos_expansion}
  a(\bx,\omega)
  =e^{\sum_{i=1}^\infty\sigma\sqrt{\lambda_i}\phi_i(\bx)\xi_i-\frac{\sigma^2}{2}
    \lambda_i\phi_i^2(\bx)}
  =\sum_{\iba\in\cJ}\frac{\Phi^{\iba}}{{\iba!}}
  \Hep_{\iba}(\ibxi),
\end{equation} 
where
$\Phi(\bx)=\left(\sigma\sqrt{\lambda_1}\phi_1(\bx),\sigma\sqrt{\lambda_2}\phi_2(\bx),\ldots\right)$.
 
From equation \eqref{eqn:exp_dia_inverse}, it can be easily
derived that
\begin{equation}\label{eqn:a_wick_inverse}
  (a^{-1}(\bx,\omega))^{\wprod(-1)}=e^{-\sigma^2} e^{\wprod\left(\sigma G(\bx,\omega)\right)}.
\end{equation}
Hence, the difference between Wiener chaos expansions of
$\left(a(\bx,\omega)^{-1}\right)^{\wprod(-1)}$ and
$a(\bx,\omega)$ is just a scaling factor $e^{-\sigma^2}$.

To make the difference between models I and II clearer, we
look at the following two linear systems
\begin{equation}\label{eqn:linear_models_13}
  \textrm{I}:\quad
  \left\{
    \begin{array}{rcl}
      \nabla u_{\sI}&=&a^{-1}*\bF_{\sI},\\
      -\nabla\cdot \bF_{\sI}&=&f,
    \end{array}
  \right.
  \qquad
  \textrm{II:}\quad
  \left\{
    \begin{array}{rcl}
      \nabla u_{\sII}&=&a^{-1}\wprod \bF_{\sII},\\
      -\nabla\cdot \bF_{\sII}&=&f.
    \end{array}
  \right.
\end{equation}
where $*$ denotes the operation of the regular product. Thus
model II is basically making the gradient ``smoother''
through the Wick product. Then the equation for
$u_{\sI}-u_{\sII}$ can be obtained as
\begin{equation}
  \left\{
    \begin{array}{rcl}
      \nabla(u_{\sI}-u_{\sII})
      &=& a^{-1}*(\bF_{\sI}- \bF_{\sII})+ a^{-1}(*-\wprod)\bF_{\sII},\\
      -\nabla\cdot(\bF_{\sI}-\bF_{\sII})&=&0,
    \end{array}
  \right.
\end{equation}
which corresponds to a second order elliptic equation for
$u_{\sI}-u_{\sII}$ as
\begin{equation}\label{eqn:u13_diff}
  -\nabla\cdot(a\nabla(u_{\sI}-u_{\sII}))
  =-\nabla\cdot\left(a*\left(a^{-1}(*-\wprod)
      \bF_{\sII}\right)\right).
\end{equation}
Note that we express explicitly the regular products on the
right-hand side since the regular and Wick products do not
commute. It is seen that equation \eqref{eqn:u13_diff}
corresponds to model I while the force term is related to
model II through $\bF_{\sII}$.

\begin{theorem}[\cite{Wan_model2}]
  \label{thm:diff_models}
  Let
  $F=-\nabla\cdot\left(a*\left(a^{-1}(*-\wprod)
      \bF_{\sII}\right)\right)$, where $*$
  indicates the regular product. Assume that
  $F \in \cR L^2(\bbF;H^{-1}(D))$, where
  $D\in\bbR^d$, $d=1,2,3$. Then there exists a set of
  weights
  $\tilde{\cR}=\{\tilde{r}_{\iba},
  \iba\in\cJ\}$, such that
  \begin{equation}
    \|u_{\sI}-u_{\sII}\|_{\tilde{\cR}L^2(\bbF;H^1_0(D))}=C(l_c)
    \sigma^2=\mathcal{O}(\sigma^2),
  \end{equation}
  where $l_c$ is the correlation length. Furthermore,
  $C(l_c)\rightarrow 0$ as $l_c\rightarrow\infty$.
\end{theorem}
\begin{remark}
  It can be shown theoretically that for one-dimensional
  cases $D\in\bbR^1$, $C(l_c)\rightarrow 0$ as
  $l_c\rightarrow 0$. For high-dimensional cases, according
  to the Landau-Lifshitz-Matheron conjecture
  {\cite{Landau60,Matheron67}} in the homogenization theory
  for log-normal random coefficients, when
  $l_c\rightarrow 0$, $C(l_c)\rightarrow \frac{1}{2}$ if
  $d=2$, and $C(l_c)\rightarrow\frac{1}{3}$ if $d=3$.
\end{remark}
\begin{remark}
  By noting the Mikulevicius-Rozovskii formula
  \cite{Boris_SNS}
  \begin{equation}\label{MB}
    \hep_{\iba}\hep_{\ibb}
    =\sum_{n=0}^\infty\frac{\cD^{n}\hep_{\iba}\wprod\cD^{n}\hep_{\ibb}}{n!},
  \end{equation}
  where $\cD^{n}$ denotes the $n$th-order Malliavin
  derivative, model I can be approximated arbitrarily well
  as
\begin{equation}\label{eqn:ellip_mr}
  -\nabla\cdot\left(
    \sum_{n=0}^\infty \frac{\cD^{n}a(\bx,\omega)\wprod\nabla\cD^{n}u}{n!}
  \right)
  =f(\bx).
\end{equation}
When $n=0$, equation \eqref{eqn:ellip_mr} recovers the
Wick-type
\begin{equation}
  -\nabla\left(a(\bx,\omega)\wprod\nabla u(\bx,\omega)\right)=f(\bx).
\end{equation}
More discussions about the new Wick-type model given by
equation \eqref{eqn:ellip_mr} can be found in
\cite{Wan_model3}.
\end{remark}

\section{Stochastic Galerkin method}\label{sec:galerkin}

\subsection{Uncertainty propagators}
We now look at the uncertainty propagator of model
I. Substituting the Wiener chaos expansion
\[
  u_{\sI}(\bx,\omega)\approx
  \sum_{\iba\in\cJ_{M,p}} u_{\sI,\iba}(\bx) \Hep_{\iba}(\ibxi)
\] 
into equation \eqref{eqn:u_I} and implementing Galerkin
projection in the probability space, we obtain the
uncertainty propagator for model I as
\begin{equation}\label{eqn:u_I_semi}
  -\sum_{\iba\in\cJ_{M,p}}\nabla\cdot
  \left(
    \bE\left[a(\bx,\omega)
      \Hep_{\iba}\Hep_{\ibg}\right]\nabla u_{\sI,\iba}(\bx)
  \right)
  =f(\bx)\delta_{\mathbf{(0)},\ibg},
  \quad \forall\ibg\in\cJ_{M,p}.
\end{equation}
It is seen that all chaos coefficients in equation
\eqref{eqn:u_I_semi} are coupled together, which means that
they must be solved together. From the numerical point of
view, a proper choice would be iterative methods. Before we
look into the numerical algorithms, we now address the
properties of the matrix
$\bE\left[a(\bx,\omega)\Hep_{\iba}\Hep_{\ibg}\right]$ for
{any} $\bx\in D$.
\begin{lemma}\label{lem:B_I_positive_definite}
  For any given $\bx\in D$, the matrix
  $B_{\sI,ij}(\bx)=\bE
  \left[a(\bx,\omega)\Hep_{\iba(i)}\Hep_{\ibg(j)}\right]$ is
  symmetric and positive definite, where $a(\bx,\omega)$ is
  a log-normal random process defined in equation
  \eqref{eqn:lognormal} and $\iba,\ibg\in\cJ_{M,p}$.
\end{lemma}
\begin{proof}
  Apparently, the matrix $B_{\sI}(\bx)$ is symmetric for any
  $\bx\in D$. For any nonzero vector
  $\bc=(c_1,c_2,\ldots,c_{N_{M,p}})\neq 0$, the following
  inequality holds for any $\bx\in D$
  \begin{align*}
    \bc^TB_{\sI}(\bx)\bc
    =&{}\sum_{i,j=1}^{N_{M,p}}c_ic_j\bE
       \left[e^{\wprod \Wphi(\bx,\omega)}\Hep_{\iba(i)}\Hep_{\ibg(j)}\right]\nonumber\\
    ={}&
         \bE\left[\sum_{i,j}^{N_{M,p}}c_ic_je^{\wprod 
         \Wphi(\bx,\omega)}\Hep_{\iba(i)}\Hep_{\ibg(j)}\right]\nonumber\\
    ={}&
         \bE\left[\left(\sum_{i=1}^{N_{M,p}}\left(e^{\wprod 
         \Wphi(\bx,\omega)}\right)^{1/2}\Hep_{\iba(i)}c_i\right)^2\right]
         \geq0\nonumber,
  \end{align*}
  In other words, $B_{\sI}$ is non-negative definite.

  We subsequently show that if
  $\bc^TB_{\sI}(\bx)\bc=0$,
  then $\bc=0$.  Let
  $\bb\in\bbR^M$. It is easy to generalize
  equation \eqref{eqn:shift_hermite} to the high-dimensional
  case
\begin{align}
  \Hep_{\iba}(\ibxi+\bb)
  &=\prod_{k=1}^M\Hep_{\alpha_k}(\xi_k+b_k)
    =\prod_{k=1}^M\sum_{i=0}^{\alpha_k}\binom{\alpha_k}{i}b_k^{\alpha_k-i}\Hep_{i}(\xi_k)\nonumber\\
  &=\sum_{\ibb\leq\iba}\binom{\iba}{\ibb}
    \bb^{\iba-\ibb}\Hep_{\ibb}(\ibxi).
\end{align}
Let $\Phi(\bx)=\Phi_1(\bx)\oplus\Phi_2(\bx)$, where
\[
  \Phi_1(\bx)=({\sigma}\sqrt{\lambda_1}\phi_1(\bx),\cdots,
  {\sigma}\sqrt{\lambda_M}\phi_M(\bx)) \textrm{ and }
  \Phi_2(\bx)=({\sigma}\sqrt{\lambda_{{M+}1}}
  \phi_{M+1}(\bx),{\sigma}\sqrt{\lambda_{M+2}}\phi_{{M+2}}(\bx),\cdots).
\]
Let $\hat{\ibxi}=(\xi_{M+1},\xi_{M+2},\ldots)$. We then have
\begin{align}
  \bc^{\mathsf{T}}B_{\sI}(\bx)\bc
  &=
    \bE\left[\left(\sum_{i=1}^{N_{M,p}}\left(e^{\wprod 
    \Wphi(\bx,\omega)}\right)^{1/2}\Hep_{\iba(i)}c_i\right)^2\right]\nonumber\\
  &=\bE\left[e^{\Phi_1^{\mathsf{T}}\ibxi+\Phi_2^{\mathsf{T}}\hat{\ibxi}-\frac{1}{2}\phivar}
    \left(\sum_{i=1}^{N_{M,p}}\Hep_{\iba(i)}c_i\right)^2\right]\nonumber\\
  &=\bE\left[e^{\Phi_2^{\mathsf{T}}\hat{\ibxi}-\frac{1}{2}\phivar}\right]\bE
    \left[e^{\Phi_1^{\mathsf{T}}\ibxi}\left(\sum_{i=1}^{N_{M,p}}\Hep_{\iba(i)}c_i\right)^2\right]\nonumber\\
  &=e^{\frac{1}{2}\Phi^{\mathsf{T}}_2\Phi_2-\frac{1}{2}\phivar}e^{\frac{1}{2}
    \Phi^{\mathsf{T}}_1\Phi_1}\bE\left[\left(
    \sum_{i=1}^{N_{M,p}}\Hep_{\iba(i)}(\ibxi
    +\Phi_1)c_i\right)^2\right]\nonumber\\
  &=\bE\left[\left(\sum_{i=1}^{N_{M,p}}\sum_{\ibb\leq\iba(i)}\binom{\iba(i)}
    {\ibb}\Phi_1^{\iba(i)-\ibb}\Hep_{\ibb}(\ibxi)c_i\right)^2\right]\nonumber\\
  &=\bE\left[\left(\sum_{\ibb\in\cJ_{M,p}}\left(\sum_{\iba(i)\geq\ibb}
    \binom{\iba(i)}{\ibb}\Phi_1^{\iba(i)-\ibb}c_i\right)
    \Hep_{\ibb}(\ibxi)\right)^2\right]\nonumber\\
  &=\sum_{\beta\in\cJ_{M,p}}\left(\sum_{\iba\geq\ibb}
    \binom{\iba}{\ibb}\Phi_1^{\iba-\ibb}c_{i(\iba)}\right)^2\ibb!.\nonumber
\end{align}
If $\bc^TB_{\sI}(\bx)\bc=0$, we
have
\[
  \sum_{\iba\geq\ibb}
  \binom{\iba}{\ibb}\Phi_1^{\iba-\ibb}(\bx)c_{i(\iba)}=0,
  \quad\forall\,\ibb\in\cJ_{M,p},\,\bx\in D.
\]
We note that the matrix in the above linear system is an
upper-triangular matrix and the entries on the diagonal line
are 1. In other words, the solution of the above linear
system is $\bc=0$.  To this end, we can conclude that the
matrix $B$ is symmetric and positive definite.
\end{proof}

\begin{remark}
  In numerical computation, we often take
$$B_{\sI,ij}(\bx)=\bE\left[e^{\Phi_1(\bx)^T\ibxi-\frac{1}{2}\phivar}
  \Hep_{\iba(i)}\Hep_{\ibg(j)}\right],$$ which is
the truncated version of the matrix $B_{\sI}$ in
lemma \ref{lem:B_I_positive_definite}. From the proof of
lemma \ref{lem:B_I_positive_definite}, such a matrix is also
symmetric and positive definite.
\end{remark}

Actually
$\bE\left[a(\bx,\omega)\Hep_{\iba}\Hep_{\ibb}\right]$
can be computed exactly as in the following lemma.
\begin{lemma}\label{lem:EaHaHb}
  Let
  $a(\bx,\omega)=\exp^{\wprod}\left(\Wphi(\bx,\omega)\right)$. We
  then have
\begin{equation}
  \bE\left[a(\bx,\omega)\Hep_{\iba}\Hep_{\ibb}\right]
  =\sum_{\ibk\leq\iba\wedge 
    \ibb}\chi(\iba,\ibb,\ibk)
  \Phi^{\iba+\ibb-2\ibk}(\bx),
\end{equation}
where $(\iba\wedge\ibb)_k=\alpha_k\wedge\beta_k$,
$k=1,2,\ldots$
\end{lemma}
\begin{proof}
  First, equation \eqref{eqn:hermite_multiply_expansion} can
  be generalized straightforwardly to the multi-dimensional
  case as
\[
  \Hep_{\iba}\Hep_{\ibb}=\sum_{\ibk\leq\iba\wedge\ibb}
  \chi(\iba,\ibb,\ibk)
  \Hep_{\iba+\ibb-2\ibk}
\]
with 
\[
  \chi(\iba,\ibb,\ibk)=\frac{\iba!\ibb!}
  {\ibk!(\iba-\ibk)!(\ibb-\ibk)!}.
\]
Using equation \eqref{eqn:a_chaos_expansion}, we have
\begin{eqnarray}
  \bE\left[a(\bx,\omega)\Hep_{\iba}\Hep_{\ibb}\right]
  &=&
      \sum_{\ibg\in\cJ}\frac{\Phi^{\ibg}(\bx)}{\ibg!}\bE
      \left[\Hep_{\ibg}\Hep_{\iba}\Hep_{\ibb}\right]\nonumber\\
  &=&\sum_{\ibg\in\cJ}\frac{\Phi^{\ibg}}{\ibg!}
      \sum_{\ibk\leq\iba\wedge\ibb}
      \chi(\iba,\ibb,\ibk)\bE
      \left[\Hep_{\iba+\ibb-2\ibk}\Hep_{\ibg}\right]\nonumber\\
  &=&\sum_{\ibk\leq\iba\wedge\ibb}
      \chi(\iba,\ibb,\ibk)\Phi^{\iba+\ibb-2\ibk}\nonumber.
\end{eqnarray}
\end{proof}
\begin{remark}
  When $\iba=\ibb$, we have
\[
  \bE\left[a(\bx,\omega)\Hep_{\iba}^2\right]=\sum_{\ibk\leq\iba}
  \chi(\iba,\iba,\ibk)\Phi^{2(\iba-\ibk)}(\bx)\geq
  \chi(\iba,\iba,\iba)=\iba!.
\]
\end{remark}
\begin{remark}
  Lemma \ref{lem:EaHaHb} implies that to compute
  $\bE[a(\bx,\omega)\Hep_{\ibb(i)}\Hep_{\ibg(j)}]$
  exactly, we require the coefficients of Wiener Chaos expansion of
  $a(\bx,\omega)$ up to order $2\ibb(N_{M,p})$.
\end{remark}

We now look at the uncertainty propagators of model II.  Let
$\hat{a}(\bx,\omega)=\left(a^{-1}\right)^{\wprod(-1)}$.
Using equations \eqref{eqn:a_chaos_expansion} and
\eqref{eqn:a_wick_inverse}, the Wiener chaos expansion of
$\hat{a}(\bx,\omega)$ can {be} explicitly derived as
\begin{equation}
  \hat{a}(\bx,\omega)=\sum_{\iba\in\cJ}
  \hat{a}_{\iba}(\bx)\Hep_{\iba}(\ibxi)
  =\sum_{\iba\in\cJ}e^{-\phivar}
  \frac{\Phi^{\iba}}{{\iba!}}\Hep_{\iba}(\ibxi). 
\end{equation} 
Following the same procedure for model I, we can obtain the
uncertainty propagator of model II as
\begin{equation}\label{eqn:u_III_semi}
  -\sum_{\iba\leq\ibg}\nabla\cdot\left(
    \hat{a}_{\ibg-\iba}(\bx)\nabla
    u_{\sII,\iba}(\bx)\right)=f(\bx)\delta_{\mathbf{(0)},\ibg},
  \quad \forall\ibg\in\cJ_{M,p}.
\end{equation}
It is seen that $u_{\sII,\ibg}$ only depends on
the chaos coefficients $u_{\sII,\iba}$ with
$\iba<\ibg$, which introduces a lower-triangular
structure into the matrix
$B_{\sII,ij}(\bx)=\hat{a}_{\ibg(j)-\iba(i)}(\bx)$. In
other words, the deterministic PDEs for
$u_{\sII,\ibg}$ are naturally decoupled and can
be solved one by one. Furthermore, equation
\eqref{eqn:u_III_semi} can be rewritten as
\[
  -\nabla\cdot\left(\hat{a}_{\mathbf{(0)}}(\bx)\nabla
    u_{\sII,\ibg}(\bx)\right)=\sum_{\iba<\ibg}
  \nabla\cdot\left(\hat{a}_{\ibg-\iba}(\bx)\nabla
    u_{\sII,\iba}(\bx)\right)+f(\bx)\delta_{\mathbf{(0)},\ibg}.
\]
Thus, if we employ finite element method to solve the PDE
system \eqref{eqn:u_III_semi}, the bilinear form remains the
same for all chaos coefficients $u_{\sII,\ibg}$,
which only depends on $\hat{a}_{\mathbf{(0)}}(\bx)$.

\subsection{Finite element discretization of uncertainty
  propagators}
We now look at the finite element discretization of
uncertainty propagators of models I and II. Let
$\mathscr{T}_h$ be a family of triangulations of $D$ with
straight edges and $h$ the maximum size of the elements in
$\mathscr{T}_h$. We assume that the family is regular, in
other words, the minimal angle of all the { elements} is
bounded from below by a positive constant. We define the
finite element space as
\begin{equation*}
  V_{h,q}^K=\Big\{v\, \Big| \, v\circ F_K^{-1}\in \mathscr{P}_q(R)\Big\},
  \quad
  V_{h,q}=\Big\{v\in H^1_\Yu{0}(D)\,\Big|\, v|_K\in V_{h,q}^K,\,
  K\in \mathscr{T}_h\Big\},
\end{equation*}
where $F_K$ is the mapping function for the element $K$
which maps the reference element $R$ (for example, an
equilateral triangle or an isosceles right triangle) to the
element $K$ and $\mathscr{P}_q(R)$ denotes the set of
polynomials of degree at most $q$ on $R$. We assume that
$v|_{\partial D}=0$ {for any} $v\in V_{h,q}$. Thus,
$V_{h,q}$ is an approximation of $H^1_0(D)$ by piece-wise
polynomial functions. There exist many choices of basis
functions on the reference elements, such as $h$-type finite
elements \cite{Ciarlet2002}, spectral/$hp$ elements
\cite{KarniadakisS05,Schwab98}, etc. Let
\[
  V_{h,q}=\textrm{span}\{\theta_1(\bx),\theta_2(\bx),\ldots,\theta_{N_x}(\bx)\}\subset
  H_0^1(D),
\]
where $N_x$ is the total number of basis functions in the
finite element space $V_{h,q}$.

The truncated Wiener chaos space {$W_{M,p}$} is defined as
\begin{equation}
  W_{M,p}=\Big\{\sum_{\iba\in{\cJ}_{M,p}}c_{\iba} {\Hep}_{\iba}({\ibxi})\
  \Big|\ c_{\iba}\in\bbR\Big\},
\end{equation}

The stochastic finite element method for model I can be
formulated as follows: Find
$u_{\sI,h}\in V_{h,q} \otimes W_{M,p}$, such that for
all $v\in V_{h,q}\otimes W_{M,p}$
\begin{equation}\label{eqn:weak_form_u_I}
  \mathcal{B}_{\sI}(u_{\sI,h},v)=\mathcal{L}(v),
\end{equation}  
where the bilinear form is
\begin{equation}
  \mathcal{B}_{\sI}(v_1,v_2)=\int_D\bE\left[a(\bx,\omega)\nabla 
    v_1\cdot\nabla v_2\right]d\bx,
\end{equation}
and the linear form
\begin{equation}
  \mathcal{L}(v)=\int_D\bE[fv]d\bx.
\end{equation}

\begin{lemma}
  The stiffness matrix for the stochastic finite element
  method of model I is symmetric and positive definite.
\end{lemma}
\begin{proof}
  Consider the approximation
  \begin{equation}\label{eqn:full_expansion_u_I}
    u_{\sI,h}(\bx,\ibxi)=\sum_{\iba\in\cJ_{M,p}}u_{\sI,h,\iba}\Hep_{\iba}(\ibxi)
    =\sum_{\overset{\scriptstyle{1\leq i\leq N_x},}{\scriptstyle{\iba\in\cJ_{M,p}}}}
    u_{\sI,h,\iba,i}\theta_i(\bx)\Hep_{\iba}(\ibxi),
\end{equation}
where $u_{\sI,h,\iba,i}\neq 0$ for some $i$ and
$\iba$.  We have
\begin{align}
  \mathcal{B}_{\sI}(u_{\sI,h},u_{\sI,h})
  &=\sum_{\overset{\scriptstyle{1\leq i\leq N_x,}}
    {\scriptstyle{\iba\in\cJ_{M,p}}}}
    \sum_{\overset{\scriptstyle{1\leq j\leq N_x},}{\scriptstyle{\ibb\in\cJ_{M,p}}}}
    \int_{D}
    \Yu{u_{\sI,h,\iba,i}u_{\sI,h,\ibb,j}}
    \bE\left[a(\bx,\omega)\Hep_{\iba}\Hep_{\ibb}\right]\nabla 
    \theta_i(\bx)\cdot\nabla \theta_j(\bx)d\bx\nonumber\\
  &=\int_D\sum_{\iba,\ibb\in\cJ_{M,p}}\bE\left[a(\bx,\omega)
    \Hep_{\iba}\Hep_{\ibb}\right]\nabla u_{\sI,h,\iba}\cdot\nabla 
    u_{\sI,h,\ibb}d\bx\nonumber\\
  &=\int_D\left(\sum_{j=1}^d \partial_{x_j} \left(\mathbf{\hat{u}}_{\sI}(\bx)\right)^{\mathsf{T}}B_{\sI}(\bx)
    \partial_{x_j} \mathbf{\hat{u}}_{\sI}(\bx)\right) d\bx\nonumber,
\end{align}
where the vector $\hat{\bu}_{\sI}(\bx)$ is
defined as
$\left(\hat{\bu}_{\sI}(\bx)\right)_k
=u_{\sI,h,\iba(k)}(\bx)$,
$k=1,\ldots,N_{M,p}$. Due to the homogeneous boundary
conditions, a nonzero constant mode does not exist in the
space $V_{h,q}$.  Using Lemma
\ref{lem:B_I_positive_definite}, we know that
$\mathcal{B}_{\sI}(u_{\sI,h},u_{\sI,h})>0$,
and the conclusion follows.
\end{proof}

\subsection{Structures of stiffness matrices of the sFEM}

Based on equation \eqref{eqn:full_expansion_u_I}, we define
some matrix notations:
\begin{equation}
  \bu_{\sI}=\left[
    \begin{array}{c}
      \bu^{\sI,1}\\
      \bu^{\sI,2}\\
      \vdots\\
      \bu^{\sI,N_{M,p}}
    \end{array}
  \right],\quad
  \bu^{\sI,i}=\left[
    \begin{array}{c}
      u_{\sI,h,\iba(i),1}\\
      u_{\sI,h,\iba(i),2}\\
      \vdots\\
      u_{\sI,h,\iba(i),N_x}
\end{array}
\right],\,\, i=1,\ldots,N_{M,p}.
\end{equation}
Obviously, the total number of unknowns is
$N_x\times N_{M,p}$. The weak form \eqref{eqn:weak_form_u_I}
leads to the linear system
$A_{\sI}\bu_{\sI}=\mathbf{f}$ with the
block structure
\begin{equation}
  A_{\sI}=\left(
    \begin{array}{cccc}
      A_{\sI,11}&A_{\sI,12}&\ldots&A_{\sI,1N_{M,p}}\\
      A_{\sI,21}&A_{\sI,22}&\ldots&A_{\sI,2N_{M,p}}\\
      \vdots&\vdots&\ddots&\vdots\\
      A_{\sI,N_{M,p}1}&A_{\sI,N_{M,p}2}&\ldots&A_{\sI,N_{M,p}N_{M,p}}
    \end{array}
  \right),\quad
  \mathbf{f}=\left(
    \begin{array}{c}
      \mathbf{f}_1\\
      \mathbf{f}_2\\
      \vdots\\
      \mathbf{f}_{N_{M,p}}
    \end{array}
  \right).
\end{equation}
Considering the approximation of $a(\bx,\omega)$ as (see
equation \eqref{eqn:a_chaos_expansion})
\begin{equation}
  a^{M,\hat{p}}(\bx,\ibxi)=\sum_{\iba\in\cJ_{M,\hat{p}}}a^{M,\hat{p}}_{\iba}
  (\bx)\Hep_{\iba}=\sum_{\iba\in\cJ_{M,\hat{p}}}
  \frac{\Phi^{\iba}(\bx)}{{\iba!}}\Hep_{\iba}(\ibxi),
\end{equation}
where $\hat{p}$ is the polynomial order of the Wiener chaos
expansion.  Then the blocks $A_{\sI,ij}$ can be expressed as
\begin{equation}
  A_{\sI,ij}=\sum_{\iba\in\cJ_{M,\hat{p}}}\bE
  \left[\Hep_{\iba}\Hep_{\ibb(i)}\Hep_{\ibg(j)}\right]S_{\iba},\quad 
  i,j=1,\ldots,N_{M,p}
\end{equation}
where
\begin{equation}
  \left(S_{\iba}\right)_{ij}=\int_Da^{M,\hat{p}}_{\iba}(\bx)
  \nabla\theta_i(\bx)\cdot\nabla\theta_j(\bx)d\bx.
\end{equation}
Define matrix $C_{\iba}$ as
\begin{equation}
  \left(C_{\iba}\right)_{ij}=\bE
  \left[\Hep_{\iba}\Hep_{\ibb(i)}\Hep_{\ibg(j)}\right].
\end{equation}
Then the matrix $A_{\sI}$ can be rewritten in the
tensor-product form as
\begin{equation}
  A_{\sI}=\sum_{\iba\in\cJ_{M,\hat{p}}}C_{\iba}\otimes S_{\iba}.
\end{equation}
Then the matrix-vector multiplication of $A_{\sI}\bu_{\sI}$
can be computed in a relatively efficient way. We rewrite
the vector $A_{\sI}\bu_{\sI}$ of length $N_xN_{M,p}$ to an
$N_{M,p}$-by-$N_x$ matrix and denote such a matrix as
$[A_{\sI}\bu_{\sI}]$. Then we have
\begin{equation}
  [A_{\sI}\bu_{\sI}]=\sum_{\iba\in\cJ_{M,\hat{p}}}[S_{\iba}
  \bu^{\sI,1}\,S_{\iba}\bu^{\sI,2}\ldots
  S_{\iba}
  \bu^{\sI,N_{M,p}}]C_{\iba}^{\mathsf{T}},
\end{equation}  
where $S_{\iba}\bu^{\sI,i}$ is the $i$th column vector of an
$N_{M,p}$-by-$N_x$ matrix.

\remove{
\begin{remark}
  When we consider $hp$ finite element discretization of the
  physical space, the components of $\bu^{\sI,i}$ can be
  further grouped as boundary modes and interior modes. The
  interior modes correspond to the basis functions that are
  equal to zero on element boundaries. In other words, the
  interior modes of $\bu^{\sI,i}$ are element-wise. When
  inverting $A_{\sI}$, we can first inverse the Schur
  complement of the sub-matrix that only involves the
  interior modes, and then address the interior modes by
  element-wise local problems. Since we are mainly
  interested in the discretization in the probabilistic
  space in this work, we will only consider linear finite
  elements for physical discretization for simplicity, where
  no interior modes exist.
\end{remark}
}

\subsection{Comments on the bilinear form
  $\mathcal{B}_{\sI}$}

Using the log-normal random coefficient $a(\bx,\omega)$, we
have shown that the bilinear form
$\mathcal{B}_{\sI}(\cdot,\cdot)$ is positive
definite. However, we do not have the ellipticity here
because $a(\bx,\omega)$ is not strictly positive. Instead of
using the Lax-Milgram lemma, the existence and uniqueness of
a solution $u(\bx,\omega)\in L_2(H_0^1(D))$ can be
established by the Fernique theorem with appropriate
regularity assumptions for the covariance function of the
underlying Gaussian field \cite{Charrier_SINUM12}. The key
observation is that the random variable
$a_{\min}^{-1}(\omega)=\min_{\bx\in D}a(\bx,\omega)\in
L_p(\bbF), p>0$. From the theoretical point of view, an
inf-sup condition can be established for the continuous
bilinear form $\mathcal{B}_{\sI}(v_1,v_2)$, where
$v_1\in L_2(\bbF;H_0^1(D))$ and
$v_2\in
L_2(\hat{\bbF}=(\Omega,\cF,a^{2}_{\min}(\omega)P(d\omega));H_0^1(D))$
\cite{Charrier_SINUM12,Starkloff11}. Note here that the
measure of the probability space for test functions $v_2$ is
weighted by the random variable
$a^{2}_{\min}(\omega)$. According to theoretical
observations, one choice for the test functions can be
\[
  \left\{\frac{v}{a_{\min}(\omega)}: v\in
    L_2(\bbF;H_0^1(D))\right\}.
\] 
However, it is not clear how to deal with $a_{\min}(\omega)$
numerically. For numerical studies of model I with the
Galerkin projection, we usually choose test functions from
$v_2\in L_2(\bbF;H_0^1(D))$.  Since the stiffness matrix
$A_{\sI}$ is symmetric and positive definite, the existence
and uniqueness of solution $\bu_{\sI}$ is guaranteed. No
divergence of the solution with respect to
$L_2(\bbF;H_0^1(D))$ norm has been observed for such a
procedure.

\section{Numerical algorithms}\label{sec:algorithm}

Based on the properties of Wick product and the assumptions
of Theorem \ref{thm:diff_models}, we have the following
asymptotic results \cite{Wan_model2} for equation
\eqref{eqn:u13_diff} satisfied by $u_{\sI}-u_{\sII}$. With
respect to $\sigma$, we have the following power series
\[
  -\nabla\cdot\left(a*\left(a^{-1}(*-\wprod)
      \bF_{\sII}\right)\right)=\sigma^2\tilde{f}_2(\bx,\ibxi)
  +\sigma^3\tilde{f}_3(\bx,\ibxi)+\ldots.
\]
Substituting
\[
  a(\bx,\omega)=a_0(\bx)+\sigma a_1(\bx,\omega)+\sigma^2
  a_2(\bx,\omega)+\ldots
\]
and the following ansatz of $u_{\sI}-u_{\sII}$
\[
  u_{\sI}-u_{\sII}=\tu_0(\bx)+\sigma\tu_1(\bx,\ibxi)+
  \sigma^2\tu_2(\bx,\ibxi)+\ldots
\]
into equation \eqref{eqn:u13_diff} and comparing the
coefficients of $\sigma^i$, we obtain
\begin{eqnarray}
  -\nabla\cdot(a_0\nabla \tu_0)&=& 0,\nonumber\\
  -\nabla\cdot(a_0\nabla \tu_1)&=& \nabla\cdot(a_1\nabla \tu_0),\nonumber\\
  -\nabla\cdot(a_0\nabla \tu_2)&=& \nabla\cdot(a_2\nabla \tu_0)
                                   +\nabla\cdot(a_1\nabla \tu_1)+\tilde{f}_2(\bx,\ibxi),\nonumber\\
                               &\ldots\ldots&\nonumber,
\end{eqnarray}
which results in
\[
  \tu_0(\bx)=\tu_{1}(\bx,\ibxi)=0,\quad
  \tu_{i}(\bx,\ibxi)\neq 0,\,i=2,3,\ldots
\]
Thus, $u_{\sI}-u_{\sII}$ has the following power series
expansion with respect to $\sigma$
\begin{equation}\label{eqn:u13_sigma}
  u_{\sI}-u_{\sII}=\sigma^2\tu_2(\bx,\ibxi)+\sigma^3\tu_3(\bx,\ibxi)
  +\ldots,
\end{equation}
which holds for any $\bx\in D$. Then both the mean and
standard deviation of $u_{\sI}-u_{\sII}$ are of
$\mathcal{O}(\sigma^2)$ if they exist.

When $l_c\rightarrow\infty$, the random coefficient becomes
\begin{equation}\label{eqn:a_1rv}
  a(\bx,\omega)=e^{\sigma\xi-\frac{1}{2}\sigma^2},
\end{equation}
where $\xi\sim\mathcal{N}(0,1)$. In other words, the noise
is spatially independent. Model II becomes
\begin{equation}
  -\nabla\cdot\left((a^{-1})^{\wprod(-1)}\wprod\nabla u\right)
  = -(a^{-1})^{\wprod(-1)}\wprod\Delta u=f(\bx),
\end{equation}
which is equivalent to model I, since
\begin{equation}
  -\Delta u=a^{-1}\wprod f(\bx)=a^{-1}f(\bx).
\end{equation}
We now consider a perturbation of the coefficient given in
equation \eqref{eqn:a_1rv}
\begin{equation}\label{eqn:a_1rv_p}
  a(\bx,\omega)=e^{\sigma(1+\epsilon\phi(\bx))\xi-\frac{1}{2}\sigma^2},
\end{equation}
where $\epsilon$ is a small positive number. When
$\epsilon\rightarrow0$, $u_{\sII}\rightarrow u_{\sI}$.  We
use the random coefficient \eqref{eqn:a_1rv_p} to mimic the
case that $l_c\rightarrow\infty$.

\begin{example}
  Consider a one-dimensional exponential covariance kernel
  on $x\in[0,1]$
  \[
    K(x_1,x_2)=e^{-\frac{|x_1-x_2|}{l_c}}.
  \]
  Its eigenvalues satisfy
  \begin{equation}
    w^2=\frac{2\epsilon-\epsilon^2\lambda_i}{\lambda_i},
    \quad
    (w^2-\epsilon^2)\tan(w)-2\epsilon w=0,
  \end{equation}
  where $\epsilon=1/l_c$. Its eigenfunctions are
  \begin{equation}
    \phi_i(x)=\frac{w\cos(wx)+\epsilon\sin(wx)}
    {\sqrt{\frac{1}{2}(\epsilon^2+w^2)+(w^2-\epsilon^2)\frac{\sin(2w)}{4w}+\frac{\epsilon}{2}(1-\cos(2w))}}.
  \end{equation}
  It can be shown that as $\epsilon\rightarrow 0$,
  $w\sim \sqrt{2}\epsilon^{1/2}$, which results in that
  $\lambda_1= 1+\mathcal{O}(\epsilon)$ and
  $\phi_1(x)=1+\mathcal{O}(\epsilon)$. Thus it is reasonable
  to consider a perturbation given in equation
  \eqref{eqn:a_1rv_p} with $\epsilon=1/l_c$.
\end{example}

We here use a one-dimensional elliptic problem to examine
the random coefficient \eqref{eqn:a_1rv_p} and present a
numerical study of the convergence behavior of
$u_{\sII}\rightarrow u_{\sI}$ as $\epsilon\rightarrow 0$. In
figure \ref{fig:diff_eps_sigma} we plot the relative
difference between $u_{\sI}$ and $u_{\sII}$ defined as
\[
  \epsilon_r =
  \frac{\|u_{\sI}-u_{\sII}\|_{L_2(\Omega;H_0^1(D))}}
  {\|u_{\sI}\|_{L_2(\Omega;H_0^1(D))}}
\]
with respect to $\sigma$ and $\epsilon$. It is seen that the
dominant error takes a form
\begin{equation}
  \log(\epsilon_r)=\log(\epsilon)+2\log(\sigma)+C,
\end{equation}
i.e.,
\begin{equation}
  \epsilon_r\sim C\epsilon\sigma^2,
\end{equation}
where $C$ is a general constant. This suggests that although
model II provides a general second-order approximation of
model I, the constant before $\sigma^2$ goes to zero
linearly with respect to $1/l_c$ as $l_c$ goes to infinity.

\begin{figure}[htbp]
  \center{
    \includegraphics[width=0.7\textwidth]{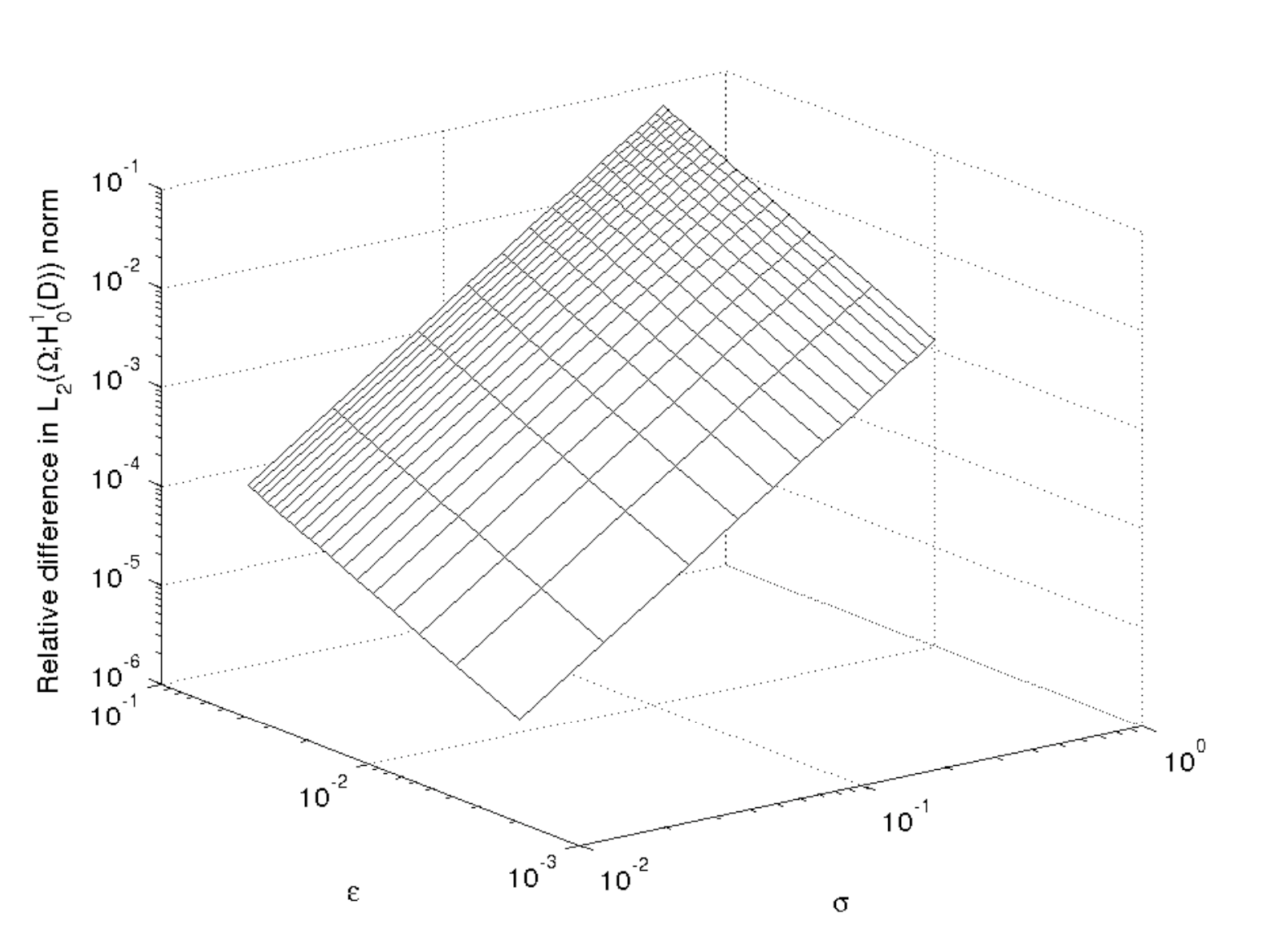}
  }
  \caption{Relative difference between $u_{\sI}$ and
    $u_{\sII}$ with respect to $\sigma$ and
    $\epsilon$ for one-dimensional elliptic problem subject
    to the random coefficient \eqref{eqn:a_1rv_p}.}
  \label{fig:diff_eps_sigma}
\end{figure}

To accelerate the numerical algorithms for model I, such as
Monte Carlo method and Galerkin projection method, we take
the advantage of the small difference between $u_{\sI}$ and
$u_{\sII}$ either when $\sigma$ is relatively small or the
correlation length is relatively large such that the
constant $C(l_c)$ is close to 0, and the fact that
$u_{\sII}$ can be obtained effectively. Based on this idea,
we use the solution $u_{\sII}$ as a predictor of $u_{\sI}$,
or the stiffness matrix \Yu{$A_{\sII}$} of model II as a
preconditioner of $A_{\sI}$.

\subsection{Variance reduction for the Monte Carlo method}

When the correlation length $l_c$ is relatively small,
eigenvalues of the covariance kernel decay slowly implying
that a relatively large number of Gaussian random variables
need to be kept for a good approximation of the log-normal
random coefficient. For such a case, the Monte Carlo method
can be more efficient than the Wiener chaos expansion. We
then propose the following two-step methodology:
\begin{enumerate}
\item[(i)] \textbf{Predictor given by $u_{\sII,h}$}:
  We first consider Wiener chaos expansion of model II to
  obtain the numerical solution $u_{\sII,h}$. Its
  mean will be just the zeroth order coefficient
  $u_{\sII,h,\boldsymbol{(0)}}$.
\item[(ii)] \textbf{A predictor-corrector method}: Using the
  solution $u_{\sII,h}$ as a control variate for variance
  reduction, we further refine the Monte Carlo simulations
  of $u_{\sI,h}$ in the following way: \Yu{
  \begin{equation}
  \tu_{\sI,h}(\bx,\ibxi)
  := u_{\sII,h,\boldsymbol{(0)}}(\bx) 
  + (u_{\sI,h}(\bx;\ibxi)-u_{\sII,h}(\bx;\ibxi)),
  \end{equation}
\begin{equation}\label{eqn:mc_Wick}
  \bE_{\text{IS}}[u_{\sI,h}](\bx) := \bE_{\tmc}[\tu_{\sI,h}](\bx)
  := 
  \frac{1}{N_{\tmc}}\sum_{i=1}^{N_{\tmc}}\tu_{\sI,h}(\bx;\ibxi^{(i)})
  ,
\end{equation}
} where $N_{\tmc}$ indicates the number of samples of
$\ibxi$ and $\ibxi^{(i)}$ the $i$-th sample.
\end{enumerate}
Based on equation \eqref{eqn:u13_sigma}, we have the
following lemma:
\Yu{
  \begin{lemma}\label{lem:IS_err}
	We have the following error estimate
	\begin{equation} \label{eqn:IS_err_estimate}
	\big\| \bE_{\text{IS}}[u_{\sI,h}] - \bE[u_{\sI,h}]  \big\|^2_{L_2(\bbF;H_0^1(D))} = 
	\int_D \var (\bE_{\text{IS}}[u_{\sI,h}])(\bx) d\bx
= \mathcal{O}(\sigma^4)N^{-1}_{\tmc}.
	\end{equation}
\end{lemma}
\begin{proof}
  Firstly, it is easy to check that
  $\bE\big[\bE_{\text{IS}}[u_{\sI,h}]\big] =
  \bE[u_{\sI,h}]$, so the first equal sign holds.  Secondly,
	\begin{align*}
	\var(\bE_{\text{IS}}[u_{\sI,h}])
	= N_\tmc^{-1} \var( \tu_{\sI,h}) 
	& = N_\tmc^{-1} \var( u_{\sI,h} - u_{\sII,h} ) \\
	& = N_\tmc^{-1} \big( \bE[ (u_{\sI,h} - u_{\sII,h})^2] - \bE^2[u_{\sI,h} - u_{\sII,h}] \big)\\
	&\leq N_\tmc^{-1} \bE[ (u_{\sI,h} - u_{\sII,h})^2]\\
	& = N_\tmc^{-1} \int (u_{\sI,h}-u_{\sII,h})^2\rho(\ibxi)d\ibxi
	=\mathcal{O}(\sigma^4)N_\tmc^{-1},\nonumber
	\end{align*}
	where the last step is obtained using
    \eqref{eqn:u13_sigma}.  Then the second equal sign of
    \eqref{eqn:IS_err_estimate} is obtained by taking
    integration of the above equation with respect to
    spatial variable $\bx$.
\end{proof}
From \eqref{eqn:IS_err_estimate}, we have
\begin{equation} \label{eqn:IS_err_estimate2}
\big\| \bE_{\text{IS}}[u_{\sI,h}] - \bE[u_{\sI,h}]  \big\|_{L_2(\bbF;H_0^1(D))}
= \mathcal{O}(\sigma^2)N^{-1/2}_{\tmc}.
\end{equation}
Since a direct Monte Carlo method to calculate
$\bE[u_{\sI,h}]$ has an error $\mathcal{O}(1) N_\tmc^{-1}$,
so the standard deviation reduction is quadratic with
respect to $\sigma$.  }


We now look at the computation cost. For the brute-force
Monte Carlo method, the cost is
$\mathcal{O}((\tau_1+\tau_2) \hat{N}_{\text{mc}})$, where
$\tau_1$ is the time for construction of the stiffness
matrix and $\tau_2$ the time for solving a linear
system. For the proposed strategy, the cost is
$\mathcal{O}((\tau_1+\tau_2+\tau_3){N}_{\text{mc}}+\tau_4)$,
where $\tau_3$ is the time for the evaluation of
$u_{\sII,h}(\bx;\ibxi^{(i)})$, which is much smaller than
$\tau_1+\tau_2$, {and $\tau_4$ is the time to obtain
  $u_{\sII,h}$}. To obtain $u_{\sII,h}$, only one stiffness
matrix is needed. Since the uncertainty propagator is
decoupled, $\tau_4\approx \tau_1+N_{M,p}\tau_2$.  Then the
cost for the proposed strategy is about
$\mathcal{O}((\tau_1+\tau_2){N}_{\text{mc}}+\tau_2N_{M,p}+\tau_1)$. Thus
if a low-order Wiener chaos solution $u_{\sII,h}$ serves as
an effective control variate, the proposed strategy can be
much more efficient than the brute-force Monte Carlo method,
since $N_{\text{mc}}$ can be much smaller than
$\hat{N}_{\text{mc}}$ for the same accuracy.

\begin{remark}
  Consider
  \begin{equation}
    \tu_{\sI,h}(\alpha;\bx,\ibxi)
    = u_{\sI,h}(\bx;\ibxi)
    - \alpha(u_{\sII,h}(\bx;\ibxi)-u_{\sII,h,\boldsymbol{(0)}}(\bx)),
  \end{equation}
  where $\alpha$ is a real number. It is well known that for
  all $\alpha\in(-\infty,\infty)$,
  $\tu_{\sI,h}(\alpha)$ provides an unbiased
  estimator of $\bE[u_{\sI,h}]$ through
  \begin{equation}
    \bE_{\tmc}[\tu_{\sI,h}]=\frac{1}{N_{\tmc}}
    \sum_{i=1}^{N_{\tmc}}\tu_{\sI,h}(\alpha;\bx,\ibxi^{(i)}),
\end{equation}
which holds for any $\bx\in D$. For a fixed $\bx\in D$, we
know that if we choose
$\alpha^*=\frac{\sigma_{\sI,\sII}}{\sigma_{\sI}^2}$
with
\[
  \sigma_{i}=\bE[(u_{i,h}-\bar{u}_{i,h})^2]^{1/2},\,
  i=\sI,\,\sII\quad
  \textrm{and}\quad
  \sigma_{\sI,\sII}=\bE[(u_{\sI,h}-\bar{u}_{\sI,h})(u_{\sII,h}-\bar{u}_{\sII,h})],
\]
the variance of $\tu_{\sI,h}$ is minimized with
respect to $\alpha$ such that
\[
  \var(\tu_{\sI,h})(\alpha^*)=\sigma_{\sI}^2(1-\rho_{\sI,\sII})^2,
\]
where
$\rho_{\sI,\sII}=\sigma_{\sI,\sII}/(\sigma_{\sI}\sigma_{\sII})$
is the autocorrelation function of $u_{\sI,h}$ and
$u_{\sII,h}$. Due to the fact given by equation
\eqref{eqn:u13_sigma} and theorem \ref{thm:diff_models},
$\rho_{\sI,\sII}\approx 1$ for small $\sigma$ or large
$l_c$, when $u_{\sI,h}$ and $u_{\sII,h}$ are almost linear
corresponding to $\alpha^*\approx 1$ (see more numerical
experiments in \cite{Wan_model2}). This is the reason we
choose $\alpha=1$ in equation \eqref{eqn:mc_Wick}.
\end{remark}
\begin{algorithm}\label{alg:MC_Wick}
  {\textbf{Solve} model II to obtain the Wiener chaos expansion of $u_{\sII,h}(\bx,\ibxi)$.}\\
  \For{$i=1,2,\ldots,N_{\tmc}$} { \textbf{Sample} model I to
    obtain $u_{\sI,h}(\bx,\ibxi^{(i)})$\; \textbf{Sample}
    the solution of model II to obtain
    $u_{\sII,h}(\bx,\ibxi^{(i)})$\; \textbf{Update} the
    statistics using an unbiased estimator as equation
    \eqref{eqn:mc_Wick}.  }
  \caption{Variance reduction for Monte Carlo simulations}
\end{algorithm}

\subsection{Stochastic Galerkin projection method}
Due to the large number of unknowns and the strong coupling
between the chaos coefficients $u_{\sI,\iba}$,
iterative numerical methods are more appropriate for solving
the linear system given by the finite element discretization
of uncertainty propagator \eqref{eqn:u_I_semi} of model
I. In other words, an effective preconditioner is
required. Consider the linear system
\begin{equation}
  A_{\sI}\bu_{\sI}=\mathbf{f}.
\end{equation}
Let $\bu_{\sII}$ be a vector consisting of unknowns from the
discretization of $u_{\sII,h}$ based on the same basis as
that for $u_{\sI,h}$. Define $A_{\sII}$ as the stiffness
matrix corresponding to the discretization of uncertainty
propagator of model II. Then the stochastic finite element
method for model II has the following matrix form
\begin{equation}\label{eqn:linear_sys_I}
  A_{\sII}\bu_{\sII}=\mathbf{f}.
\end{equation}
Based on structure of the uncertainty propagator of model
II, we know that $A_{\sII}$ is a block lower triangular
matrix
\begin{equation}
  A_{\sII}=\left(
    \begin{array}{cccc}
      A_{\sII,11}& 0 &\ldots& 0\\
      A_{\sII,21}&A_{\sII,22}&\ldots& 0\\
      \vdots&\vdots&\ddots&\vdots\\
      A_{\sII,N_{M,p}1}&A_{\sII,N_{M,p}2}&\ldots&A_{\sII,N_{M,p}N_{M,p}}
    \end{array}
  \right),
\end{equation}
where the blocks $A_{\sII,ij}$ is defined as
\begin{equation}
  A_{\sII,ij}=S_{\ibg(i)-\iba(j)},\quad i\geq j.
\end{equation}
with
\begin{equation}
  (S_{\ibg(i)-\iba(j)})_{m,n}=\int_D\hat{a}_{\ibg(i)-\iba(j)}(\bx)\nabla\theta_m(\bx)
  \cdot\nabla\theta_n(\bx)d\bx.
\end{equation}
Note that
\begin{equation}
  A_{\sII,11}=A_{\sII,22}=\ldots=A_{\sII,N_{M,p},N_{M,p}}=S_{\boldsymbol{(0)}}.
\end{equation}
\begin{lemma}
  Consider the stiffness matrices $A_{\sI}$ and
  $A_{\sII}$. We have that the condition number
  \begin{equation}
    \kappa\left(A_{\sII}^{-1}A_{\sI}\right)\leq 1+\mathcal{O}(\sigma^2).
  \end{equation}
\end{lemma}
\begin{proof}
  Since the difference between $u_{\sI}$ and $u_{\sII}$ is
  of $\mathcal{O}(\sigma^2)$, we have in the matrix form
  \begin{equation}
    \|\bu_{\sI}-\bu_{\sII}\|=
    \|A_{\sI}^{-1}\mathbf{f}-A_{\sII}^{-1}\mathbf{f}\|=\mathcal{O}(\sigma^2),
\end{equation}
which holds for any $\mathbf{f}$. Hence
\begin{equation}
  \|A_{\sI}^{-1}-A_{\sII}^{-1}\|=\mathcal{O}(\sigma^2).
\end{equation}

Then the condition number of
$A_{\sII}^{-1}A_{\sI}$ is
\begin{align}
  \kappa&=\|A_{\sII}^{-1}A_{\sI}\|\|A_{\sI}^{-1}A_{\sII}\|\nonumber\\
        &=\|\left(A_{\sII}^{-1}-A_{\sI}^{-1}+A_{\sI}^{-1}\right)A_{\sI}\|
          \|\left(A_{\sI}^{-1}-A_{\sII}^{-1}+A_{\sII}^{-1}\right)A_{\sII}\|\nonumber\\
        &=\|I+\left(A_{\sII}^{-1}-A_{\sI}^{-1}\right)A_{\sI}\|
          \|I+\left(A_{\sI}^{-1}-A_{\sII}^{-1}\right)A_{\sII}\|\nonumber\\
        &\leq 1 + \|A_{\sI}\|\|A_{\sII}\|\left(\mathcal{O}(\sigma^2)+\mathcal{O}(\sigma^4)\right).
\end{align}
\end{proof}
\begin{remark}
  When $\sigma$ is relatively small, we expect that
  $A_{\sII}$ can provide a good preconditioner for linear
  system \eqref{eqn:linear_sys_I}. Instead of solving
  equation \eqref{eqn:linear_sys_I}, we can solve
  \begin{equation}
    A_{\sII}^{-1}A_{\sI}\bu_{\sI}=A_{\sII}^{-1}\mathbf{f}.
  \end{equation}
\end{remark}

\subsubsection{Preconditioned Richardson's iteration}
One commonly used iterative method for the uncertainty
propagator \eqref{eqn:u_I_semi} of model I is the block
Gauss-Seidel method, which can be expressed as
\begin{eqnarray}
  &&-\nabla\cdot\left(
     \bE\left[a(\bx,\omega)\Hep_{\ibg}^2\right]
     \nabla u_{\ibg}^{\sI,n+1}(\bx)\right)\nonumber\\
  &=&\sum_{i=1}
      ^{k(\ibg)-1}\nabla\cdot\left(
      \bE\left[a(\bx,\omega)\Hep_{\iba(i)}\Hep_{\ibg}\right]
      \nabla u_{\iba(i)}^{\sI,n+1}(\bx)\right)\\
  &+&\sum_{i=k(\ibg)+1}^{N_{M,p}}\nabla\cdot\left(
      \bE\left[a(\bx,\omega)\Hep_{\iba(i)}\Hep_{\ibg}\right]
      \nabla u_{\iba(i)}^{\sI,n}(\bx)\right)+f(\bx)\delta_{
      \boldsymbol{(0)},\ibg},\quad\forall\, \ibg\in\cJ_{M,p},\nonumber
\end{eqnarray}
where the superscript $n$ indicates the iteration step. It
is shown in Lemma \ref{lem:EaHaHb} that
$\bE\left[a(\bx,\omega)\Hep_{\ibg}^2\right]$ is strictly
positive. We know that the block Gauss-Seidel method
corresponds to a fixed point iteration on a preconditioned
system
\[
  M^{-1}A_{\sI}\bu_{\sI}=M^{-1}\mathbf{f},
\]
where $M$ is the lower-triangular part of matrix
$A_{\sI}$. Based on the comparability of models I and II, we
can construct the following preconditioned Richardson's
iterative method \cite{Saad03}:
\begin{equation}
  \bu_{\sI}^{(k+1)}=\bu_{\sI}^{(k)}+\gamma 
  A_{\sII}^{-1}(A_{\sI}\bu_{\sI}^{(k)}-\mathbf{f}),
\end{equation}
where $\gamma$ is the non-negative acceleration
parameter. We know that the Richardson's iterative method
converges when $\gamma<2/\rho(A_{\sII}^{-1}A_{\sI})$, where
$\rho(\cdot)$ indicates the spectral radius of a matrix.
Based on the relation between $A_{\sI}$ and $A_{\sII}$, we
expect that $\rho((A_{\sII})^{-1}A_{\sI})$ is close to 1
when $\sigma$ is relatively small.

\subsubsection{Preconditioned GMRES method}
We {also} consider Krylov subspace methods. Since $A_{\sI}$
is symmetric and positive definite, a common choice to solve
the linear system is preconditioned Conjugate Gradient (CG)
method. We here consider to use $A_{\sII}$ as a
preconditioner, which is not symmetric. Hence we use a
preconditioned GMRES method \cite{Saad03} instead of CG
method.

\section{Numerical results}\label{sec:results}
We consider both one-dimensional and two-dimensional
($D = [-1,1]^d, d=1,2$) elliptic problem with random
coefficient subject to a non-zero force term
\begin{equation}
  f(\bx) = \prod_{i=1}^d (x_i^2+4x_i+1) e^{x_i}
\end{equation}
and homogeneous boundary conditions. Assume the underlying
Gaussian random field of the log-normal coefficient
$a(\bx,\omega)=e^{\sigma G(\bx,\omega)-\frac12\sigma^2}$,
with $G$'s correlation function is given by:
\begin{equation}
  K(\bx_1,\bx_2)=e^{-\frac{|\bx_1-\bx_2|^2}{2l_c^2}},
\end{equation} 
or
\begin{equation}
  K(\bx_1,\bx_2)=e^{-\frac{|\bx_1-\bx_2|}{l_c}},
\end{equation}  
where $l_c$ being the correlation length and $\sigma$ the
standard deviation. Due to the analyticity of the Gaussian
kernel, the eigenvalues decay exponentially
\cite{Schwab_CMAME05}. The decay rate is determined by the
value of the correlation length, where a larger $l_c$
corresponds to a faster decay rate. The physical
discretization is given by $25$ uniform finite element with
order $q=4$ for the one-dimensional case, and 32$\times$32
uniform quadratic finite elements for the two-dimensional
cases.  We test the parameters $\sigma=0.2, 0.6, 1$ and
$l_c=20, 2, 0.2$.  The solution differences of model I and
model II is similar to the results in \cite{Wan_model2} and
\cite{Wan_model3}. So we only sketch the results for
two-dimensional case here.

The results for 2-dimensional case with Gaussian type kernel
are given in Fig \ref{fig:IvIInormal_L20-2d},
\ref{fig:IvIInormal_L2-2d}, \ref{fig:IvIInormal_L02-2d}
\Yu{for $l_c=20, 2, 0.2$, respectively}. The results for
2-dimensional exponential kernel \Yu{ with $l_c=20, 2, 0.2$}
are given in Fig \ref{fig:IvIIexp_L20-2d},
\ref{fig:IvIIexp_L2-2d}, \ref{fig:IvIIexp_L02-2d},
respectively. \Yu{The truncation errors of the K-L expansion
  for the Gaussian kernel and exponential kernel are set to
  be $2\times 10^{-3}$ and $3\times 10^{-2}$,
  respectively}. For
Model I, if the dimension of the stochastic space
$M$ is less than 20, we use stochastic Galerkin method, otherwise we use Monte Carlo method.  From these
figures, we say for small $\sigma$ values, the results of
Model II agree very well with the results of Model I. A
larger correlation length $\ell_c$ also makes a better
agreement between the results of Model I and Model II. This
is consistent to the theoretical results.
\begin{figure}[htp]
  \begin{centering}
    \includegraphics[width=0.35\textwidth]{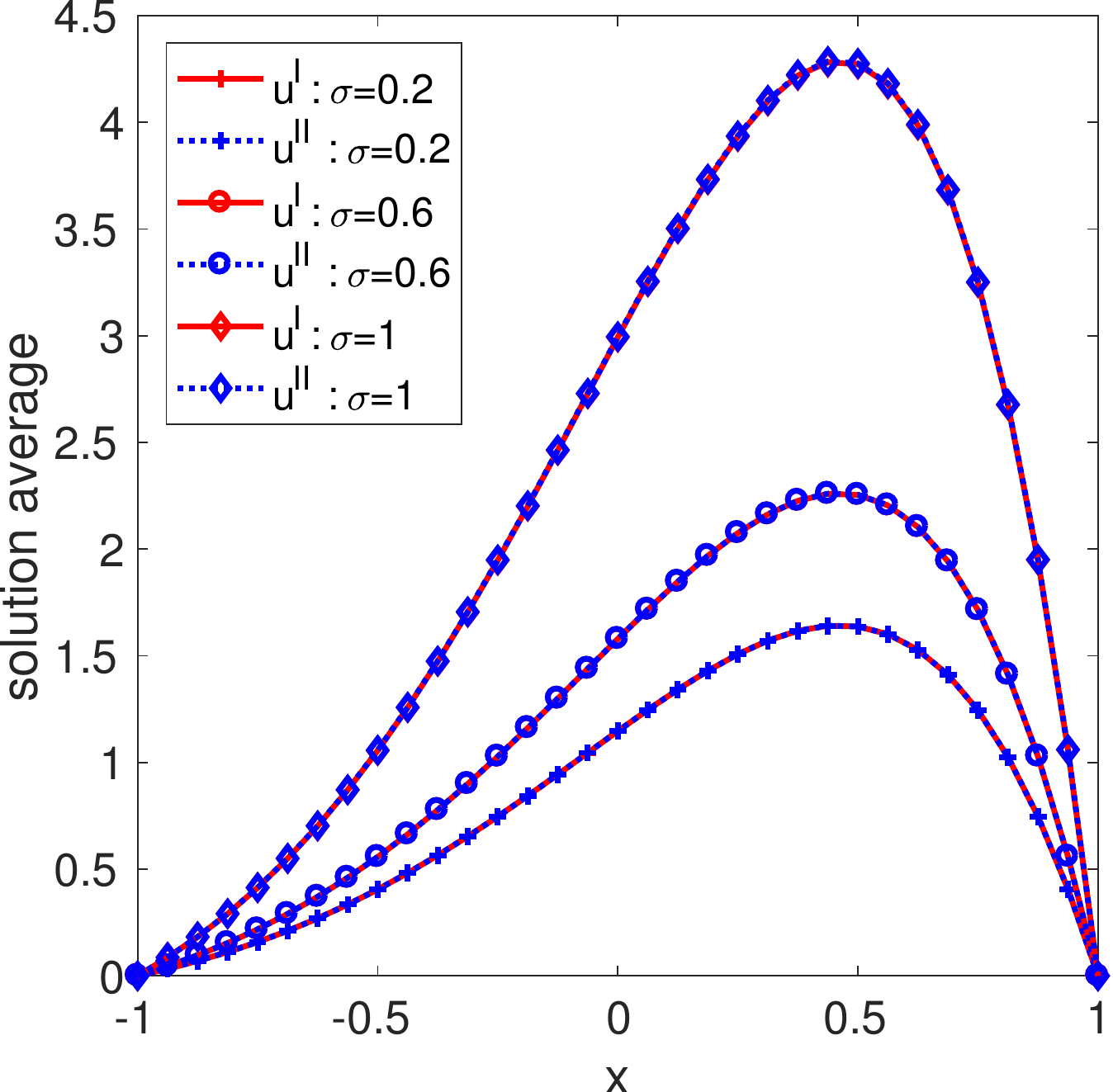}
    \hspace{4ex}
    \includegraphics[width=0.34\textwidth]{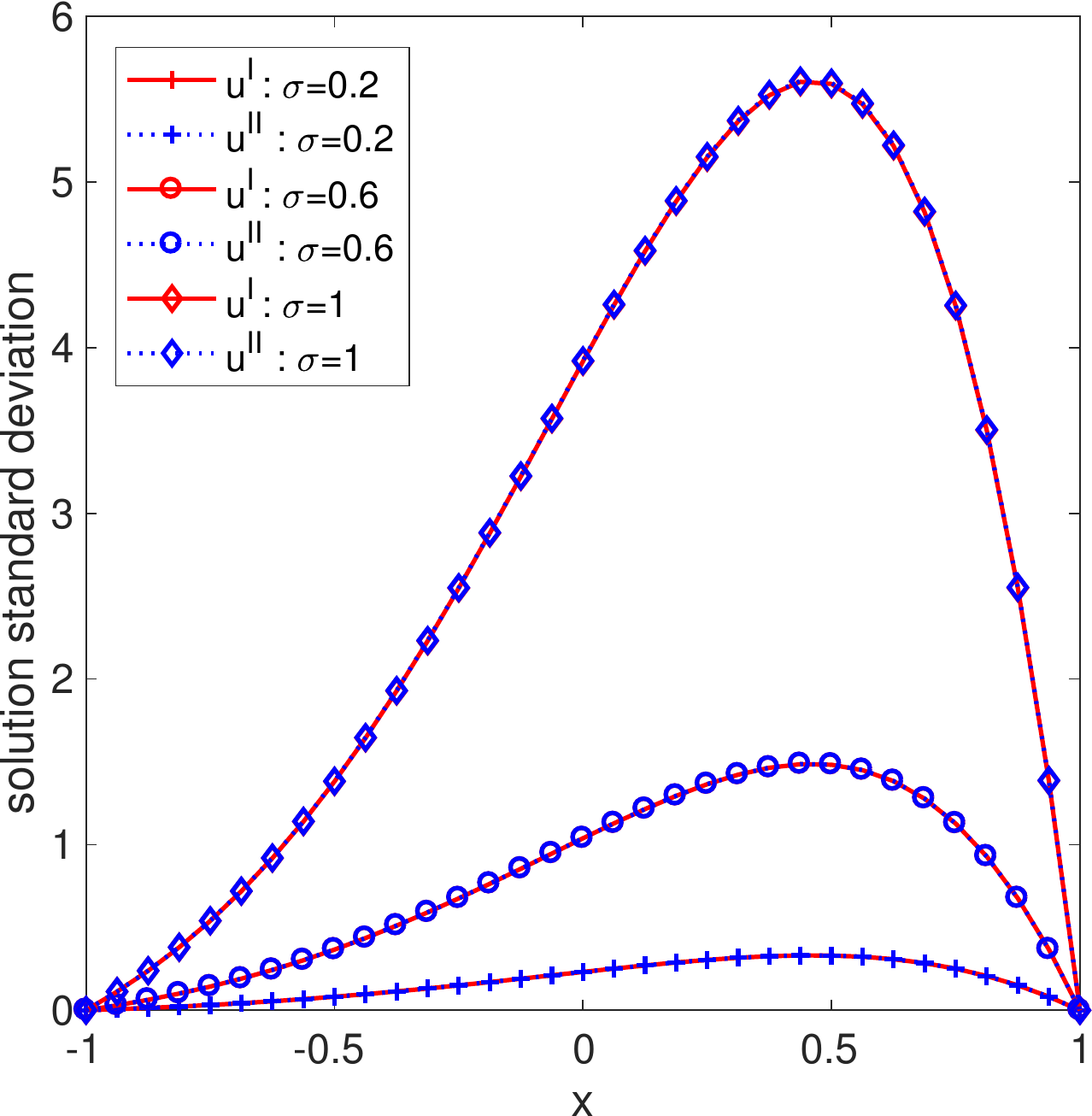}
    \par\end{centering}
  \caption{\label{fig:IvIInormal_L20-2d}The average (left)
    and standard deviation (right) of model I and II at the
    horizontal line $y=0$: Gaussian kernel with
    $\ell_{c}=20$, $M=1$, {and} $p=16$ {are} used for the stochastic
    Galerkin approximation of both Model I and Model II. }
\end{figure}
\begin{figure}[htp]
	\begin{centering}
      \includegraphics[width=0.35\textwidth]{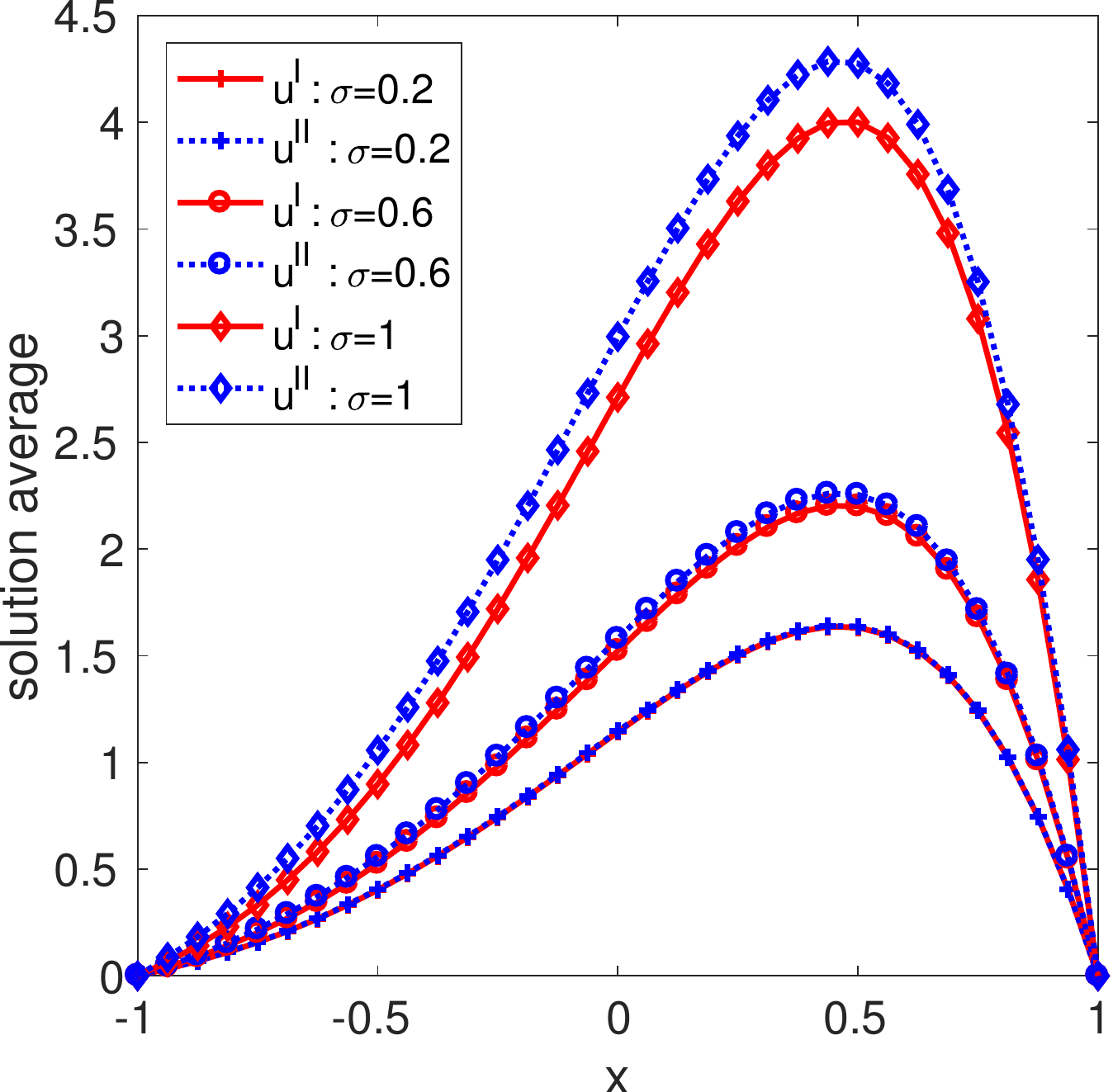}
      \hspace{4ex}
      \includegraphics[width=0.34\textwidth]{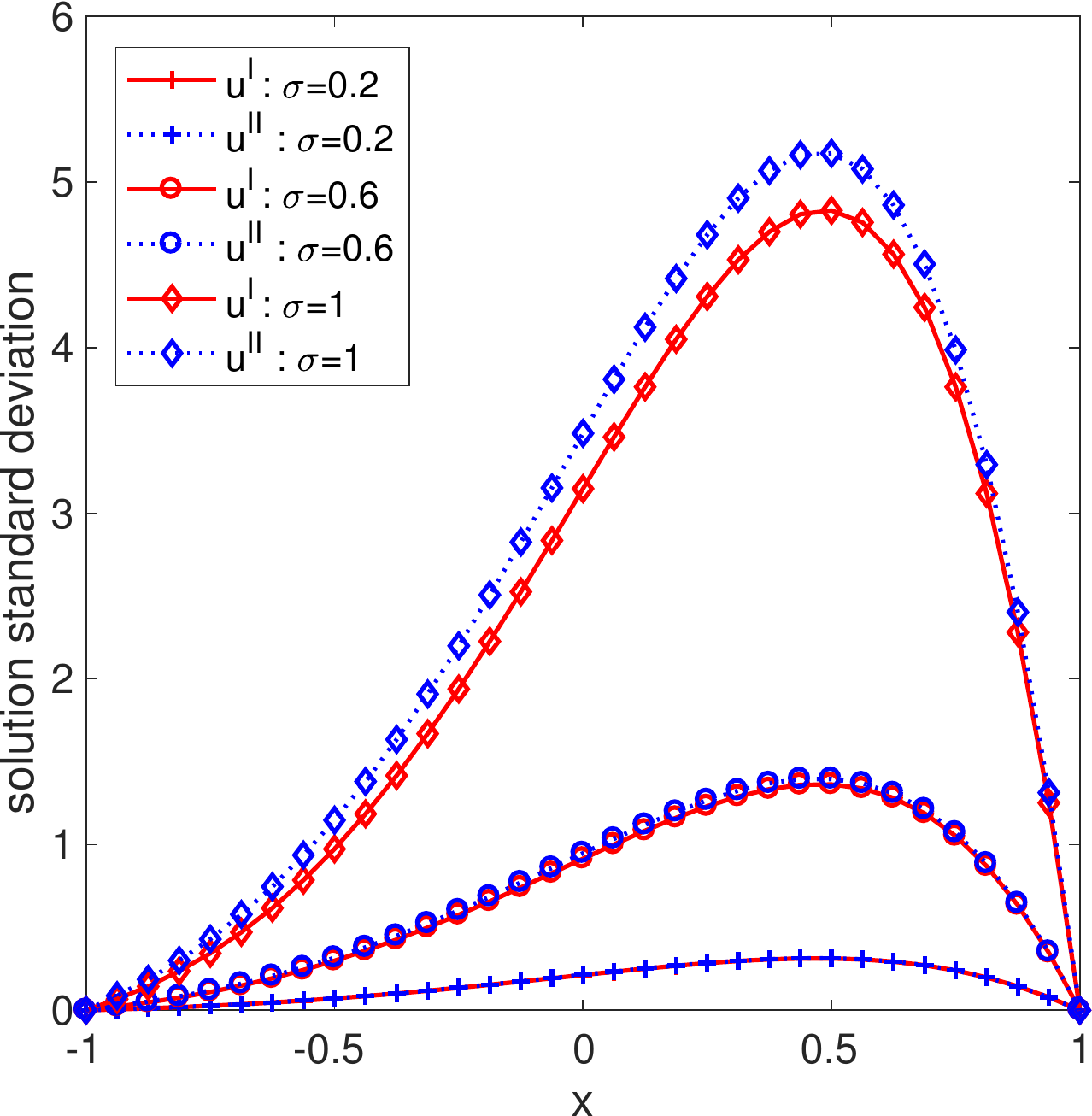}
      \par\end{centering}
	\caption{\label{fig:IvIInormal_L2-2d}The average (left)
      and standard deviation (right) of model I and II at
      the horizontal line $y=0$: Gaussian kernel with
      $\ell_{c}=2$, \Yu{$M=6$, and $p=6$} {are} used for the
      stochastic Galerkin approximation of both Model I and
      Model II. }
  \end{figure}
  \begin{figure}[htp]
	\begin{centering}
      \includegraphics[width=0.35\textwidth]{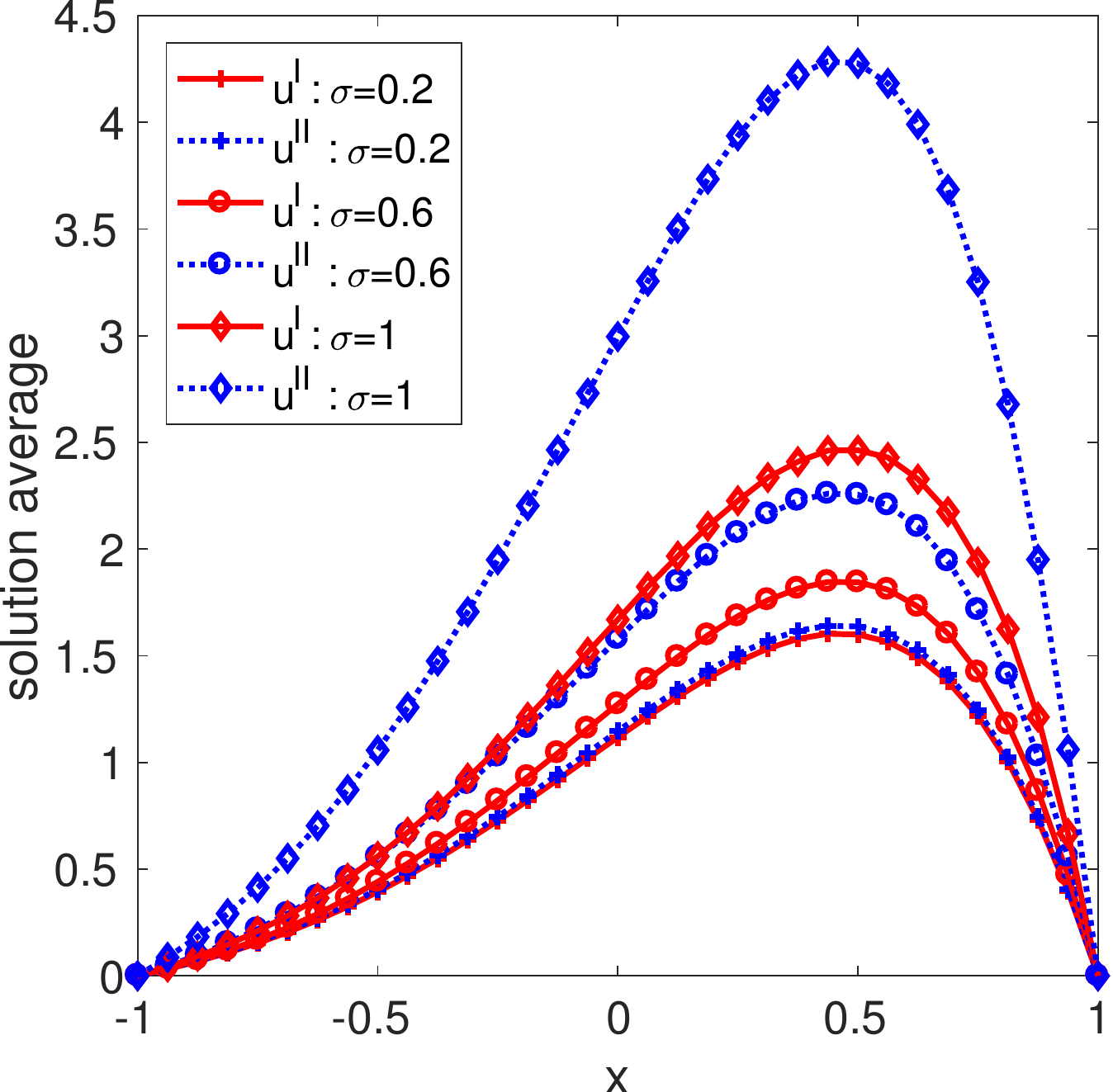}
      \hspace{4ex}
      \includegraphics[width=0.35\textwidth]{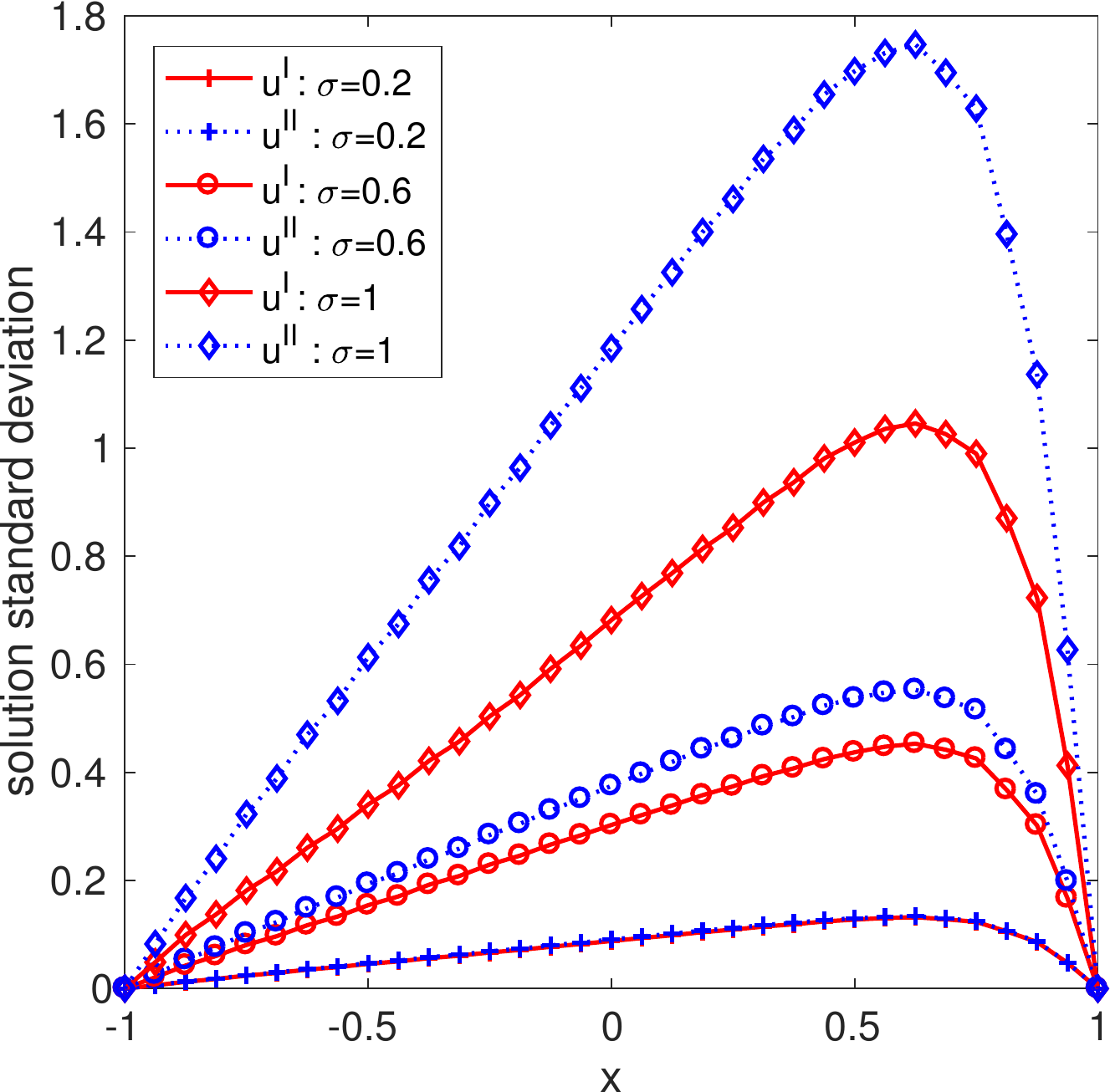}
      \par\end{centering}
	\caption{\label{fig:IvIInormal_L02-2d}The average (left)
      and standard deviation (right) of model I and II at
      the horizontal line $y=0$: Gaussian kernel with
      $\ell_{c}=0.2$, $\Yu{M=94}$, and $p=1$ {are} used for the
      stochastic Galerkin approximation of Model II.
      $\Yu{M=94}$ and $N_{mc}=10000$ {are} used for the Monte Carlo
      method of model I.}
  \end{figure}

  \begin{figure}
	\begin{centering}
      \includegraphics[width=0.35\textwidth]{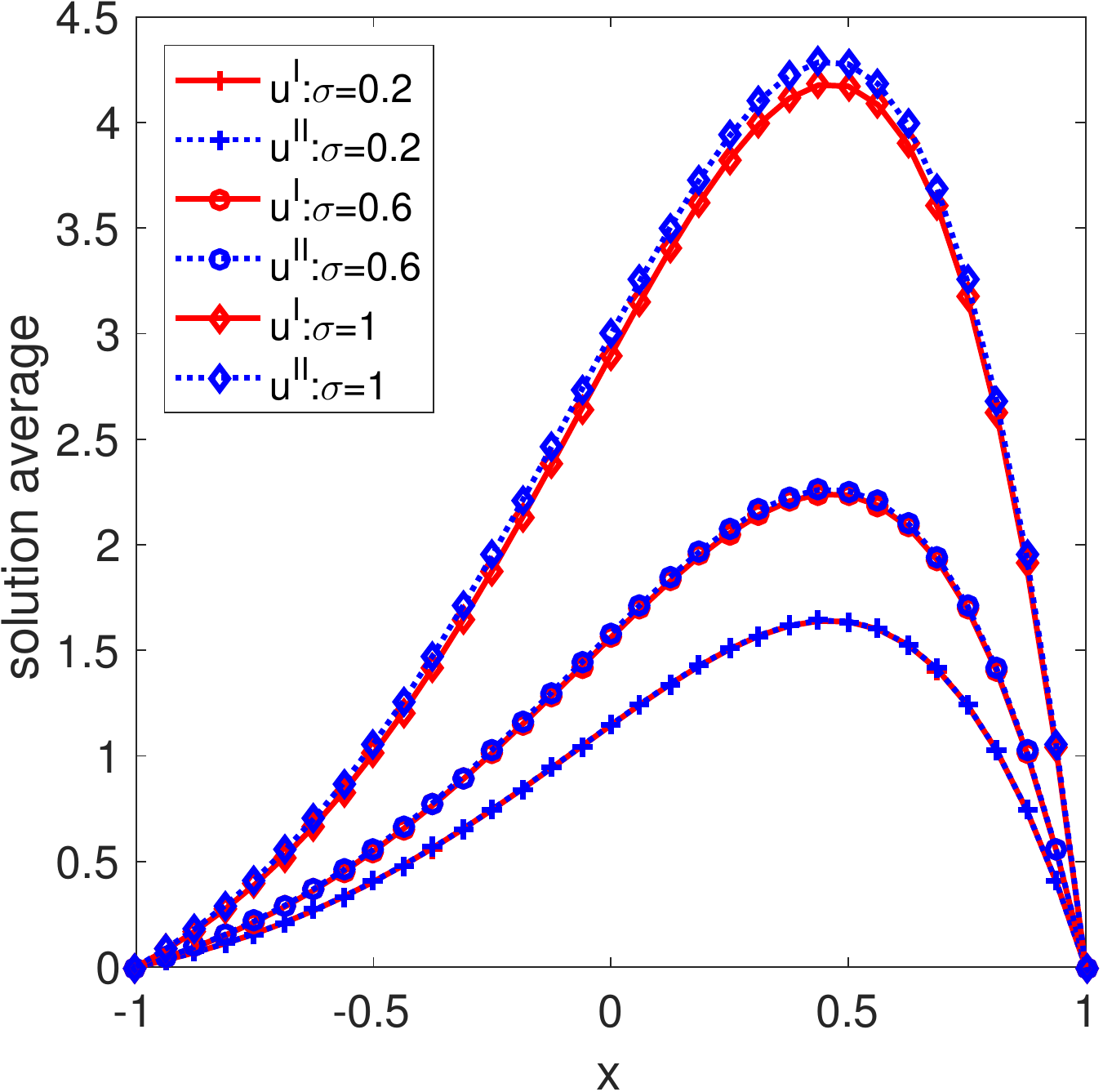}
      \hspace{4ex}
      \includegraphics[width=0.34\textwidth]{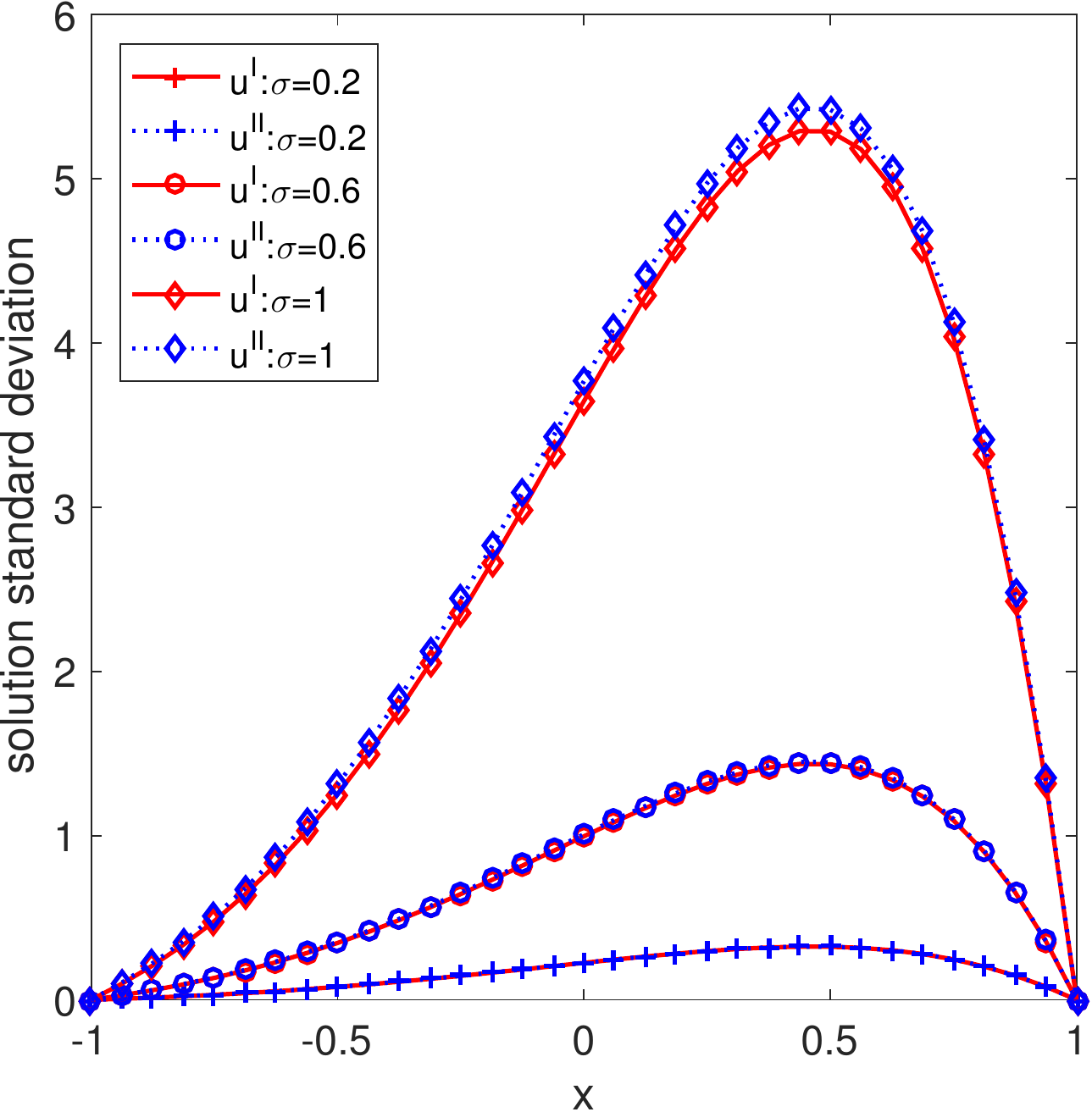}
      \par\end{centering}
	\caption{\label{fig:IvIIexp_L20-2d}The average (left)
      and standard deviation (right) of model I and II at
      the horizontal line $y=0$: exponential kernel with
      $\ell_{c}=20$, \Yu{$M=3$, and $p=8$} {are} used for the
      stochastic Galerkin approximation of both Model I and
      Model II.}
  \end{figure}
  \begin{figure}
	\begin{centering}
      \includegraphics[width=0.35\textwidth]{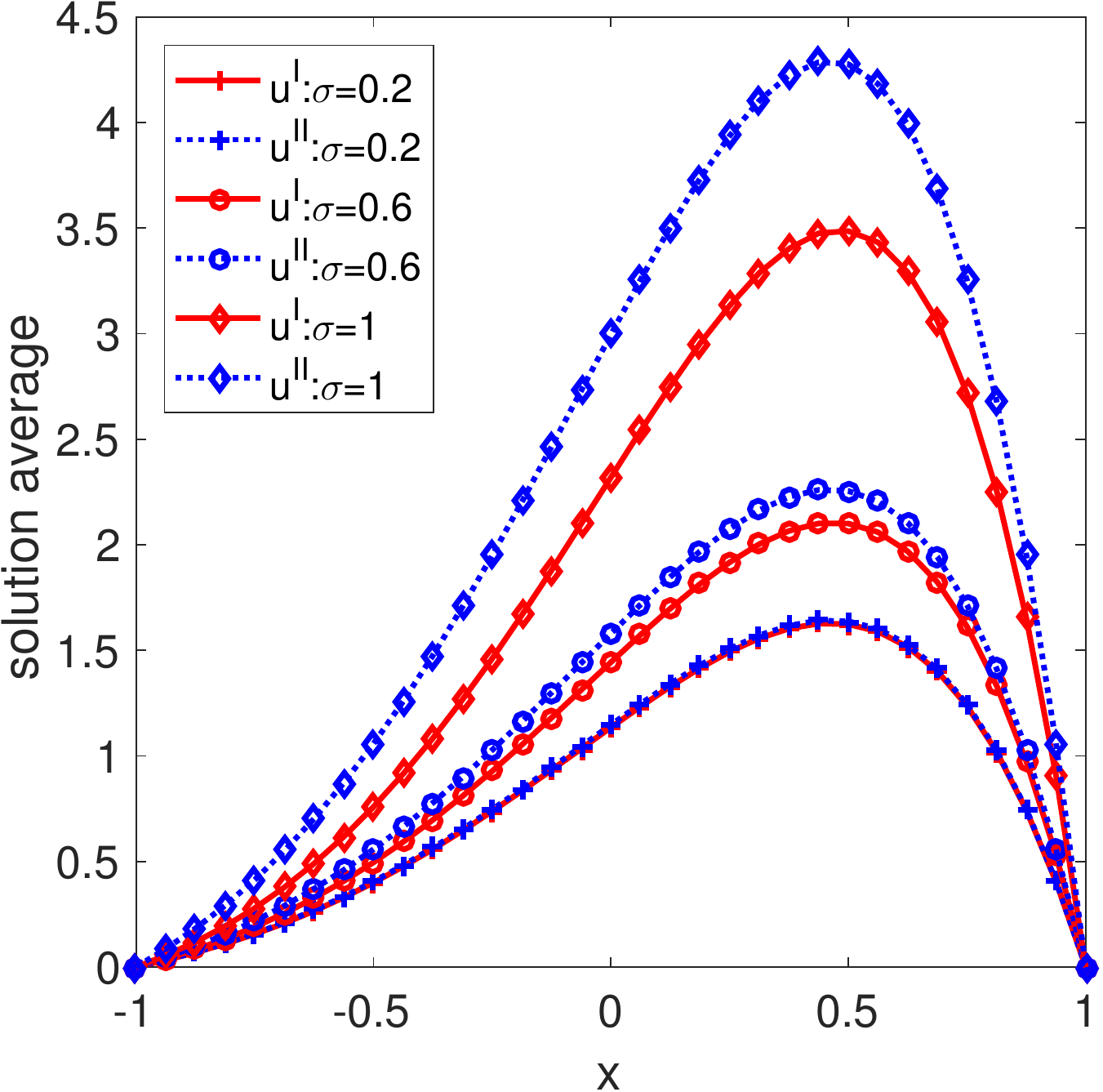}
      \hspace{4ex}
      \includegraphics[width=0.35\textwidth]{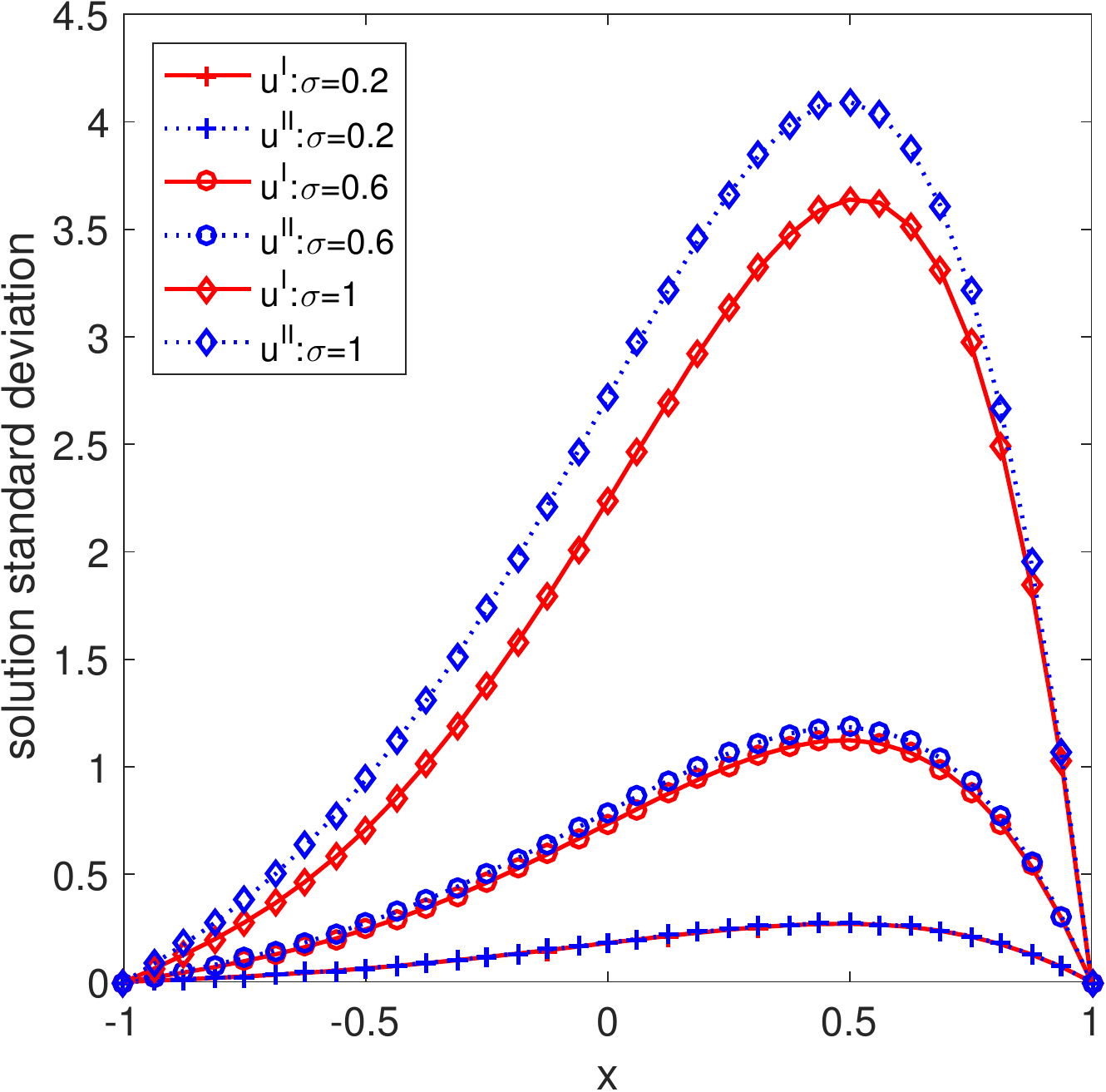}
      \par\end{centering}
	\caption{\label{fig:IvIIexp_L2-2d}The average (left) and
      standard deviation (right) of model I and II at the
      horizontal line $y=0$: exponential kernel with
      $\ell_{c}=2$, \Yu{$M=28$, and $p=2$} {are} used for the
      stochastic Galerkin approximation of Model II.
      $M=\Yu{28}$ and $N_{mc}=10000$ {are} used for the Monte Carlo
      method of model I.}
  \end{figure}
  \begin{figure}
	\begin{centering}
      \includegraphics[width=0.34\textwidth]{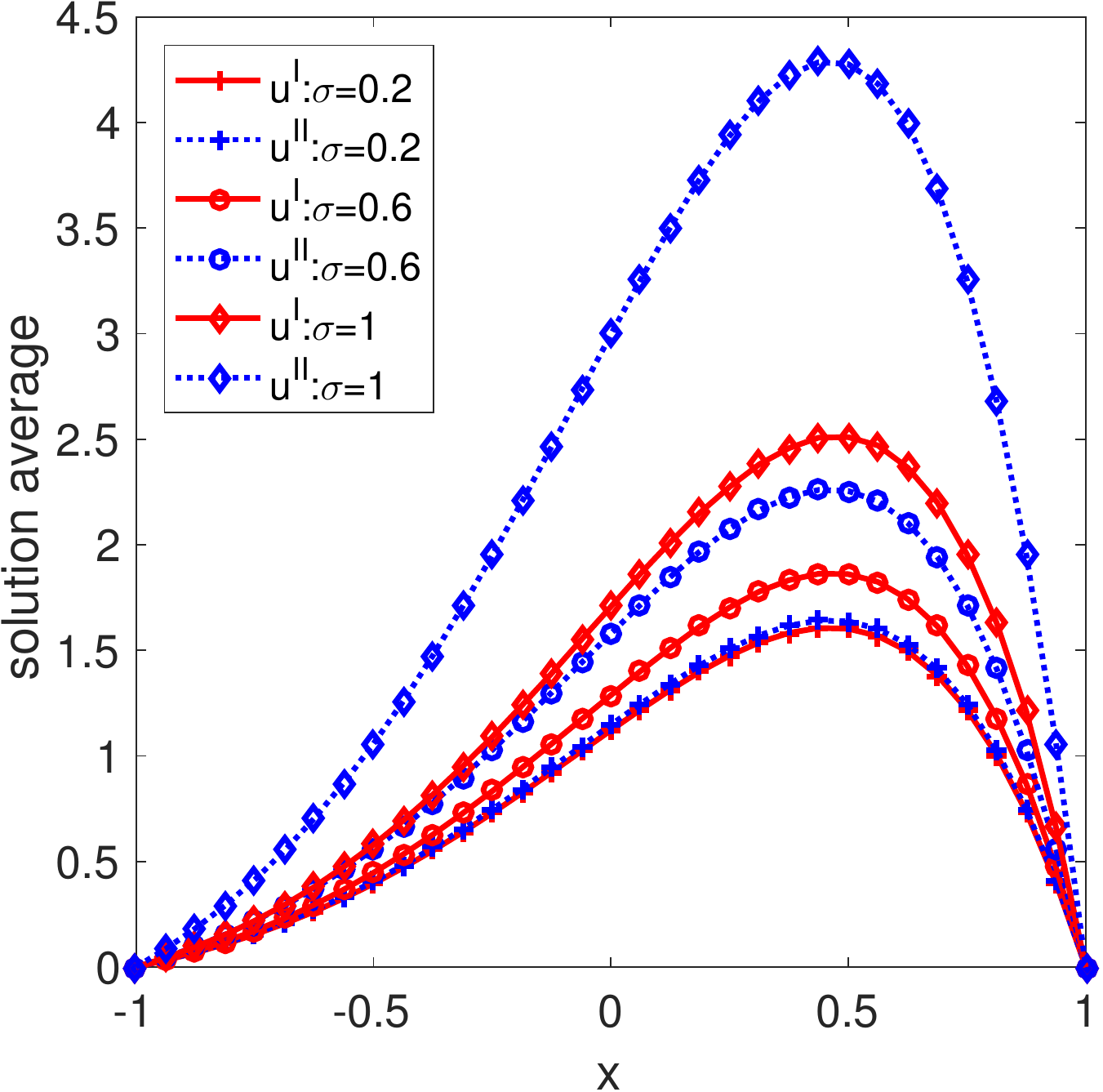}
      \hspace{4ex}
      \includegraphics[width=0.34\textwidth]{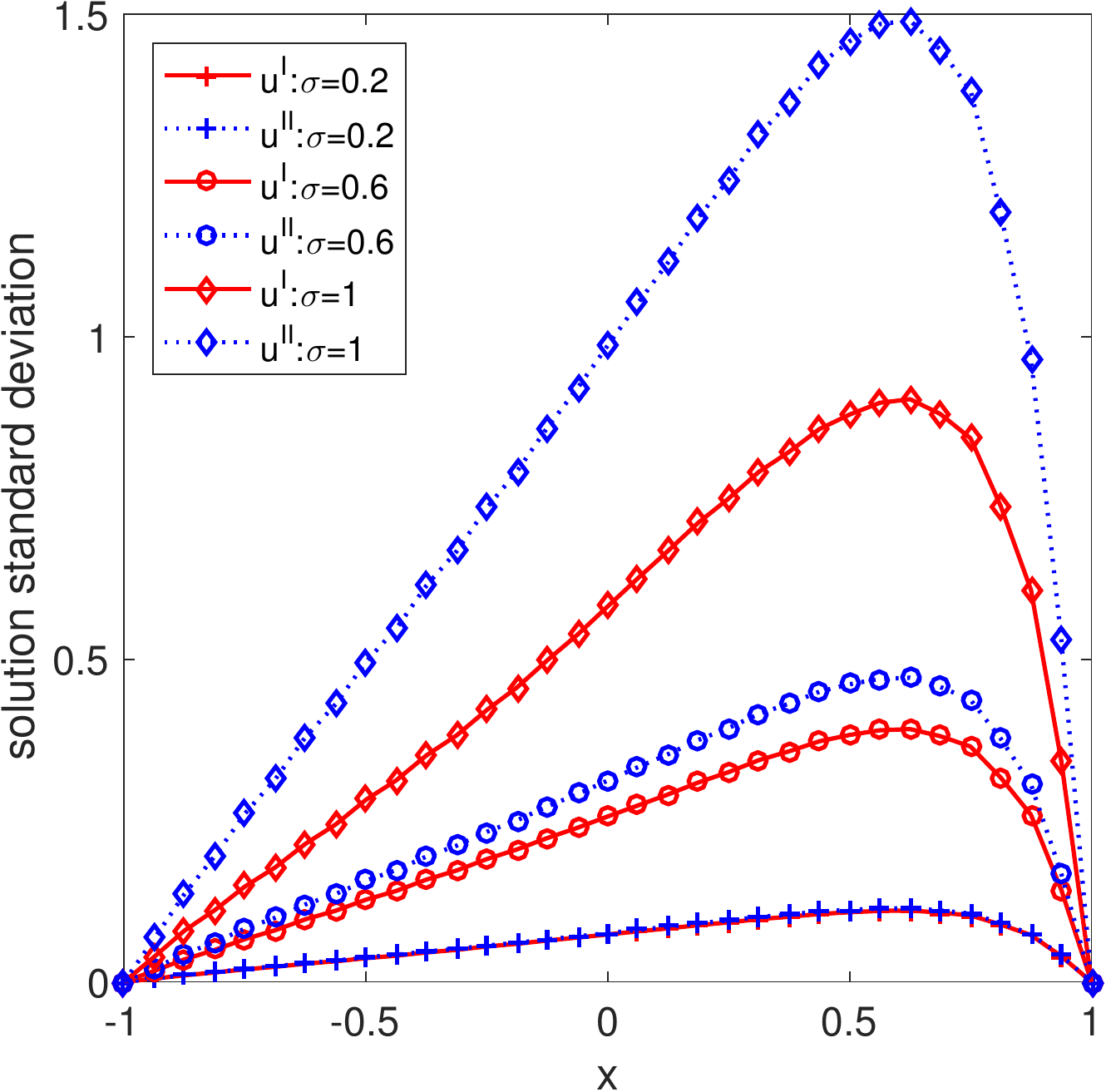}
      \par\end{centering}
	\caption{\label{fig:IvIIexp_L02-2d}The average (left)
      and standard deviation (right) of model I and II at
      the horizontal line $y=0$: exponential kernel with
      $\ell_{c}=0.2$, $M=\Yu{86}$ and $p=1$ {are} used for the
      stochastic Galerkin approximation of Model II.
      $M=\Yu{86}$ and $N_{mc}=10000$ {are} used for the Monte Carlo
      method of model I.}
  \end{figure}

  \subsection{Using $u_{\sII,h}$ as a control
    variate}

  When the correlation length is relatively small, a large
  number of random variables are required to represent the
  random coefficient and the Monte Carlo method would be a
  better choice for computation.  The mean and variance are
  given by the following unbiased estimators, respectively:

  \begin{align*}
    \bar{u}_{\sI,h}
    & =\frac{1}{N_{\tmc}}\sum_{i=1}^{N_{\tmc}}
      u_{\sI,h}(\bx,\ibxi^{(i)}),\\
    \var(u_{\sI,h})
    & \approx \frac{1}{N_{\tmc}-1}\sum_{i=1}^{N_{\tmc}}
      (u_{\sI,h}(\bx,\ibxi^{(i)})-\bar{u}_{\sI,h}(\bx))^2.
  \end{align*}

  \begin{figure}
    \begin{centering}
      {\includegraphics[width=0.4\textwidth]{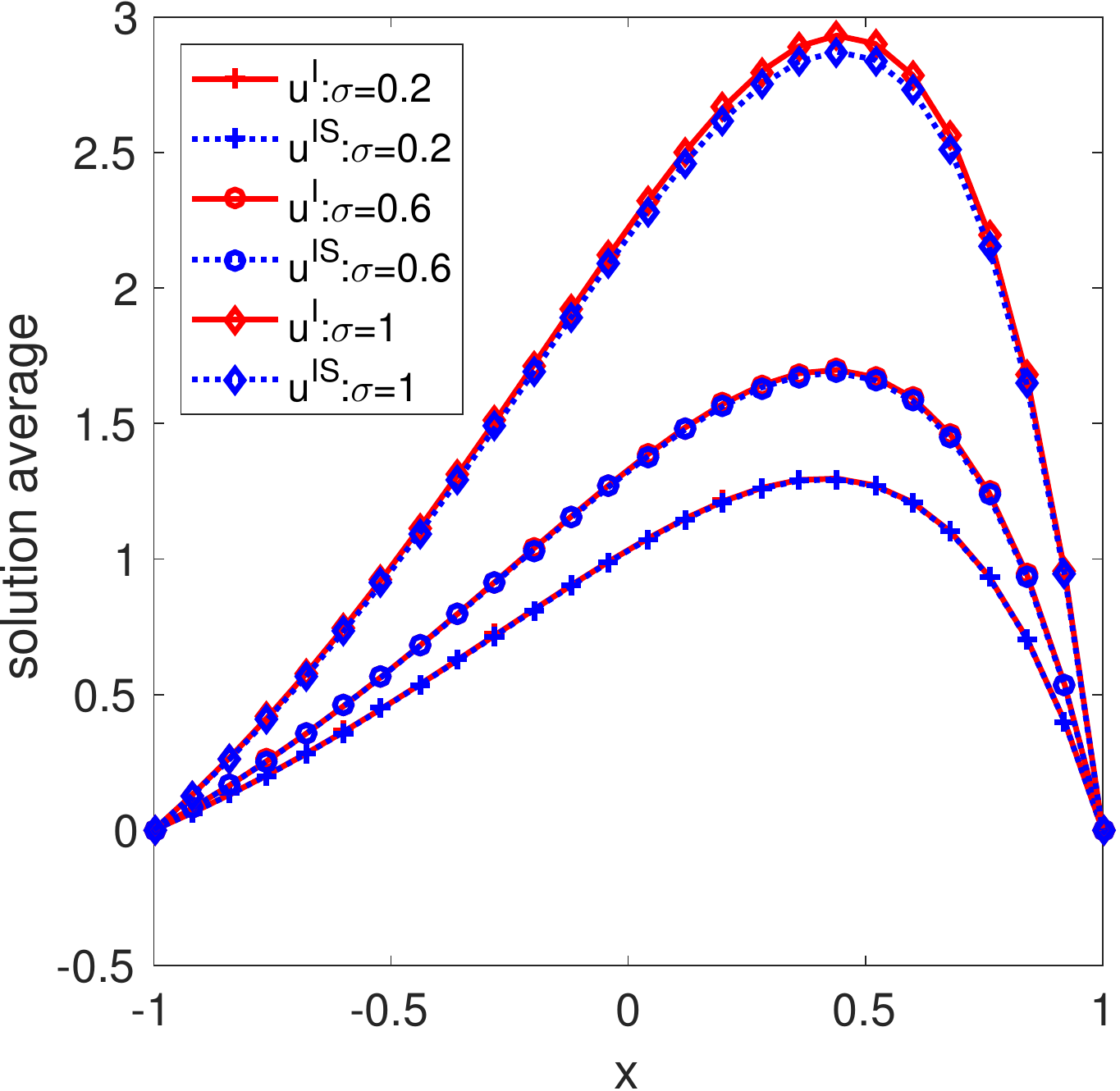}}$\quad$
      {\includegraphics[width=0.4\textwidth]{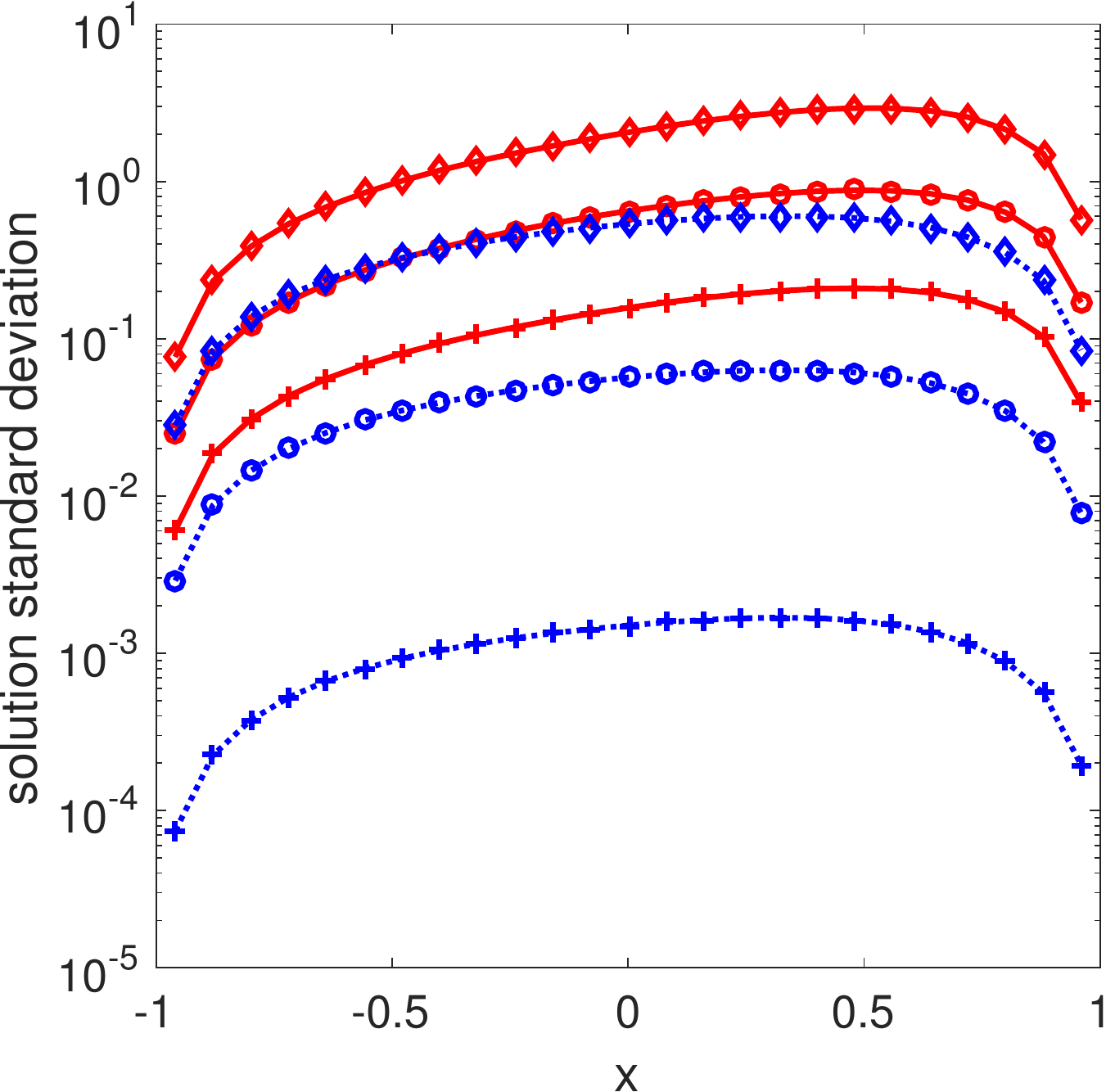}}\par
    \end{centering}
    \caption{\label{fig:IS_exp-1d} \Yu{The mean and standard
        deviation of the Monte Carlo method for model I with
        and without important sampling in 1-dimensional
        case. The exponential kernel with correlation length
        $l_c=1$ is used. $M=12, p=4$ for the stochastic
        Galerkin approximation of Model II.
        $M=12, N_{mc}=10000$ for the Monte Carlo method.
        Note that $\log$ scale is used for the standard
        deviation.} }
  \end{figure}
  \begin{figure}
	\begin{centering}
      {\includegraphics[width=0.4\textwidth]{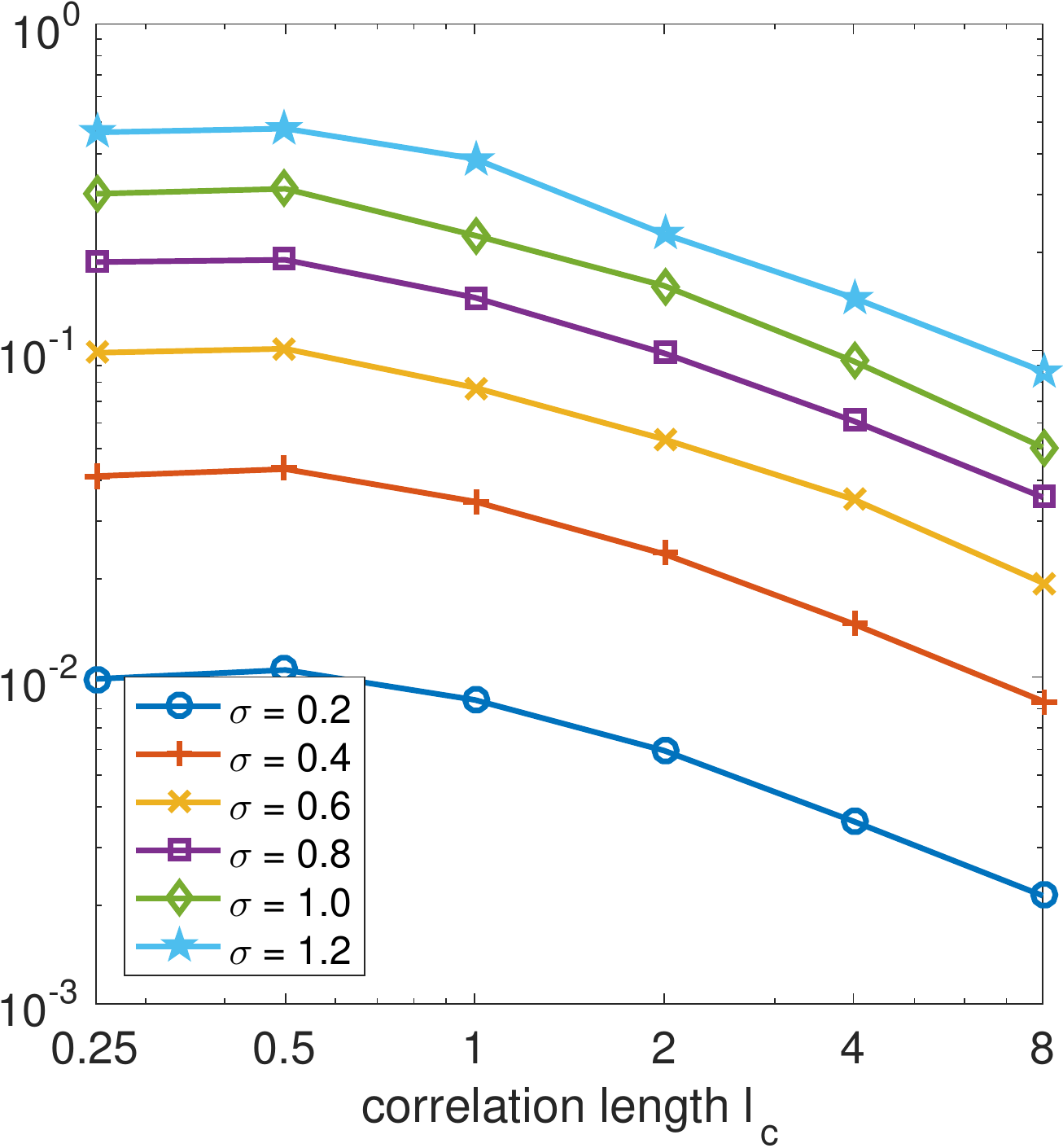}}$\quad$
 	  {\includegraphics[width=0.4\textwidth]{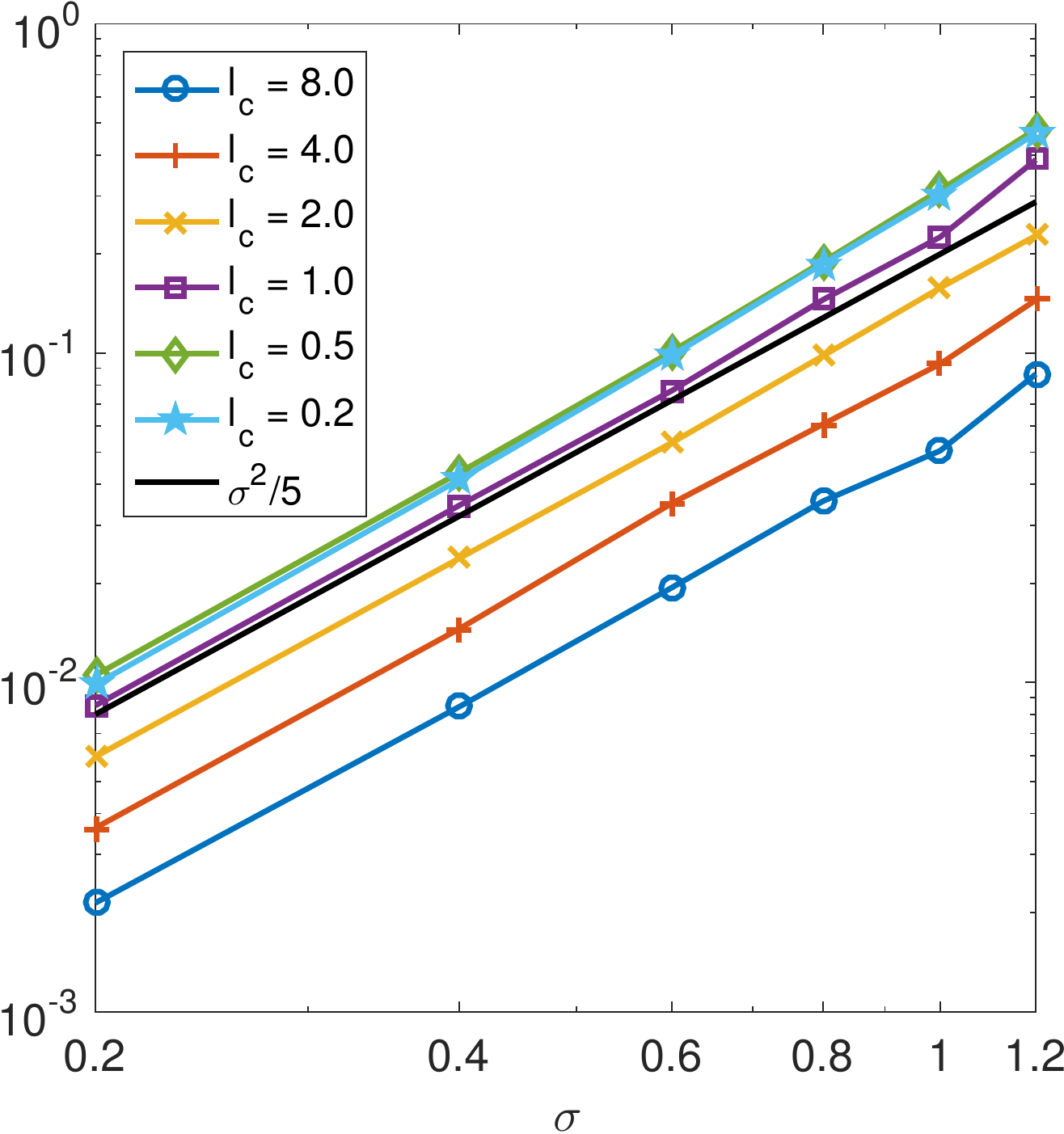}}
 	  \par
	\end{centering}
	\caption{\label{fig:IS_exp-1d_VR} \Yu{The variance
        reduction for the 1-dimensional case with
        exponential kernel having different correlation
        length and different values of $\sigma$ . The $y$-axes are
        $\|\var(\tilde{u}_{I,h})\|_{H_0^1(D)}/\|\var(u_{I,h})\|_{H_0^1(D)}$.
        $N_{mc}=10000$ samples are used for the Monte Carlo
        method.  The tolerance of K-L expansion is set to
        $3\times 10^{-2}$.  The values of $M, p$
        corresponding to the stochastic Galerkin
        approximation of Model II with
        $l_c=8, 4, 2, 1, 0.5, 0.25$ are
        $(3,6), (4,6), (7,5), (12,4), (19, 3), (27,3)$, respectively.
        Note that $\log$ scales are used for both $x$ and
        $y$ axes.}  }
  \end{figure}
  \begin{figure}
	\begin{centering}
      {\includegraphics[width=0.4\textwidth]{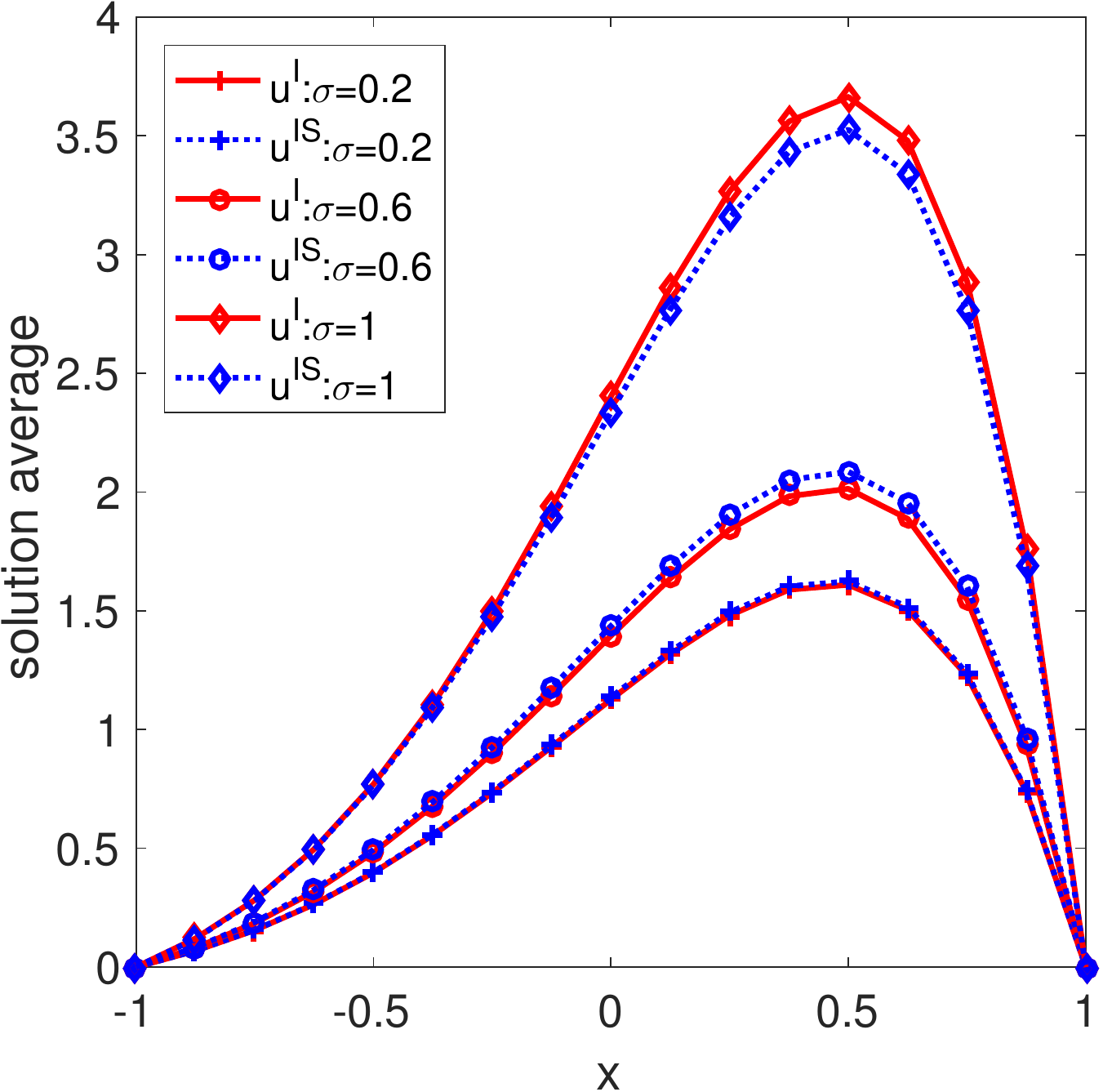}}
$\quad$
	{\includegraphics[width=0.4\textwidth]{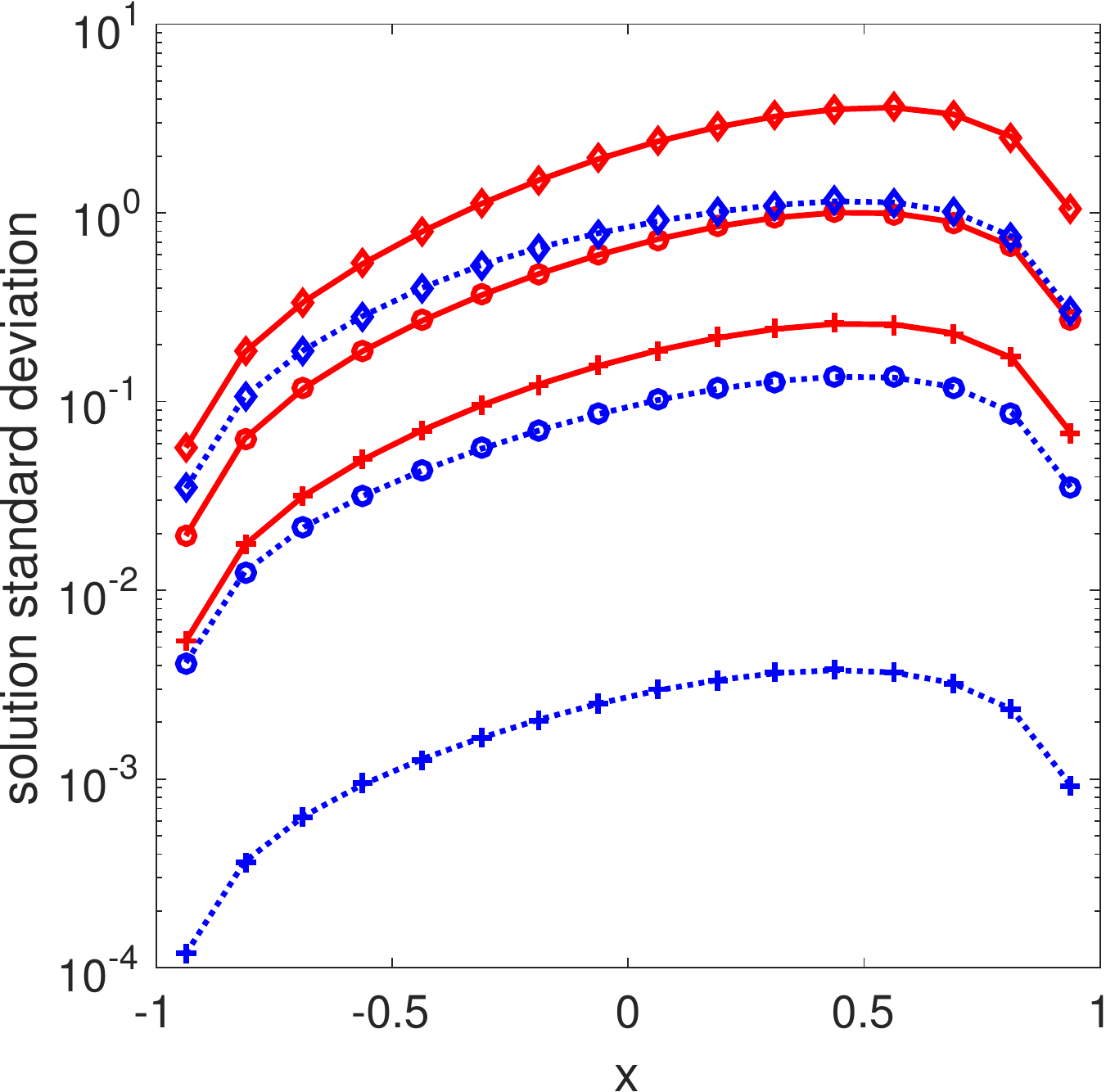}}\par
	\end{centering}
	\caption{\label{fig:IS_exp-2d} \Yu{The mean and standard
        deviation of the Monte Carlo method for model I with
        and without important sampling in 2-dimensional
        case. The exponential kernel with correlation length
        $l_c=2$ is used. $M=19, p=2$ for the stochastic
        Galerkin approximation of Model II.
        $M=19, N_{mc}=1000$ for the Monte Carlo method.
        Note that $\log$ scale is used for the standard
        deviation.}}
  \end{figure}
  \begin{figure}
	\begin{centering}
    {\includegraphics[width=0.4\textwidth]{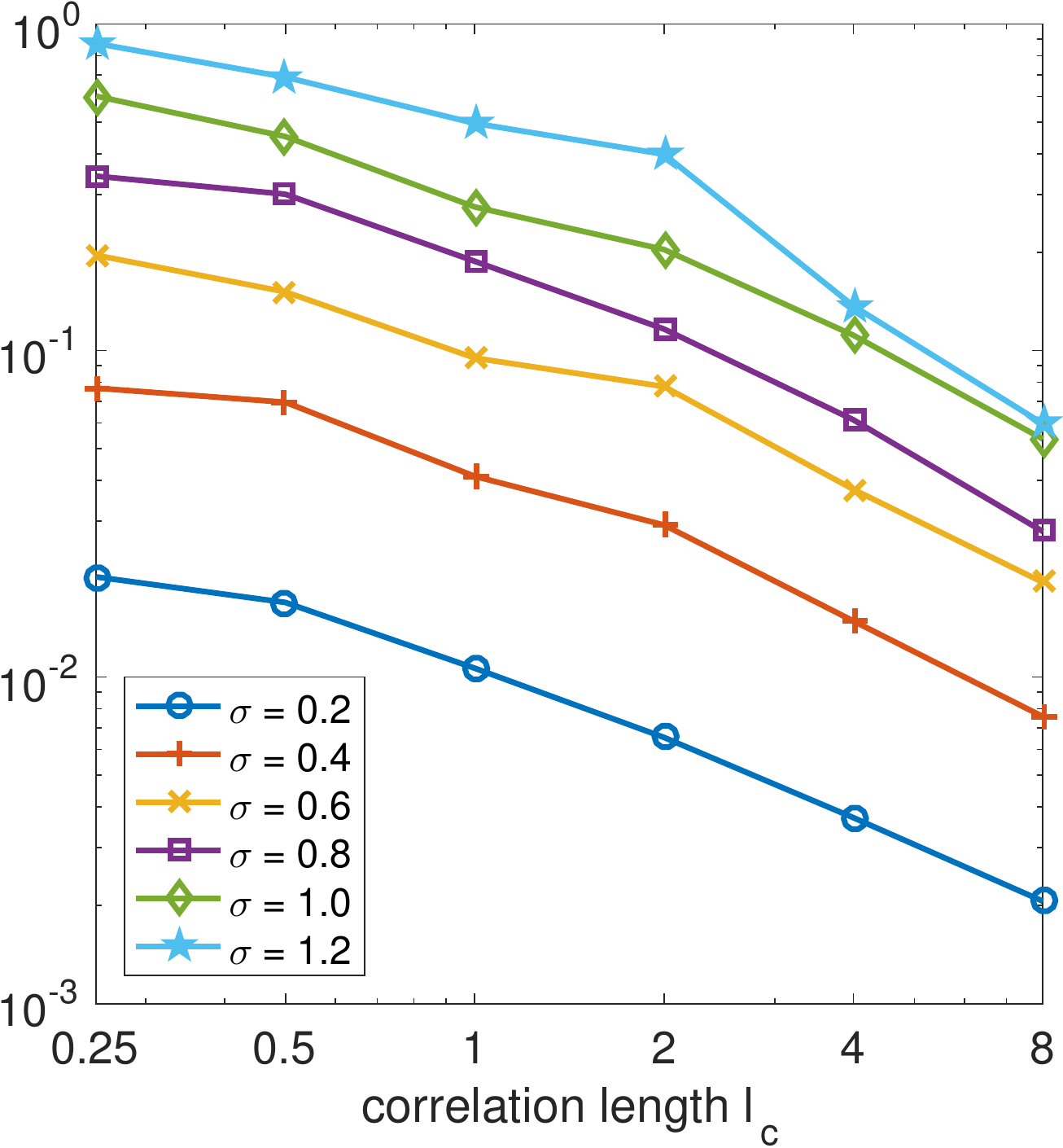}}$\quad$
    {\includegraphics[width=0.4\textwidth]{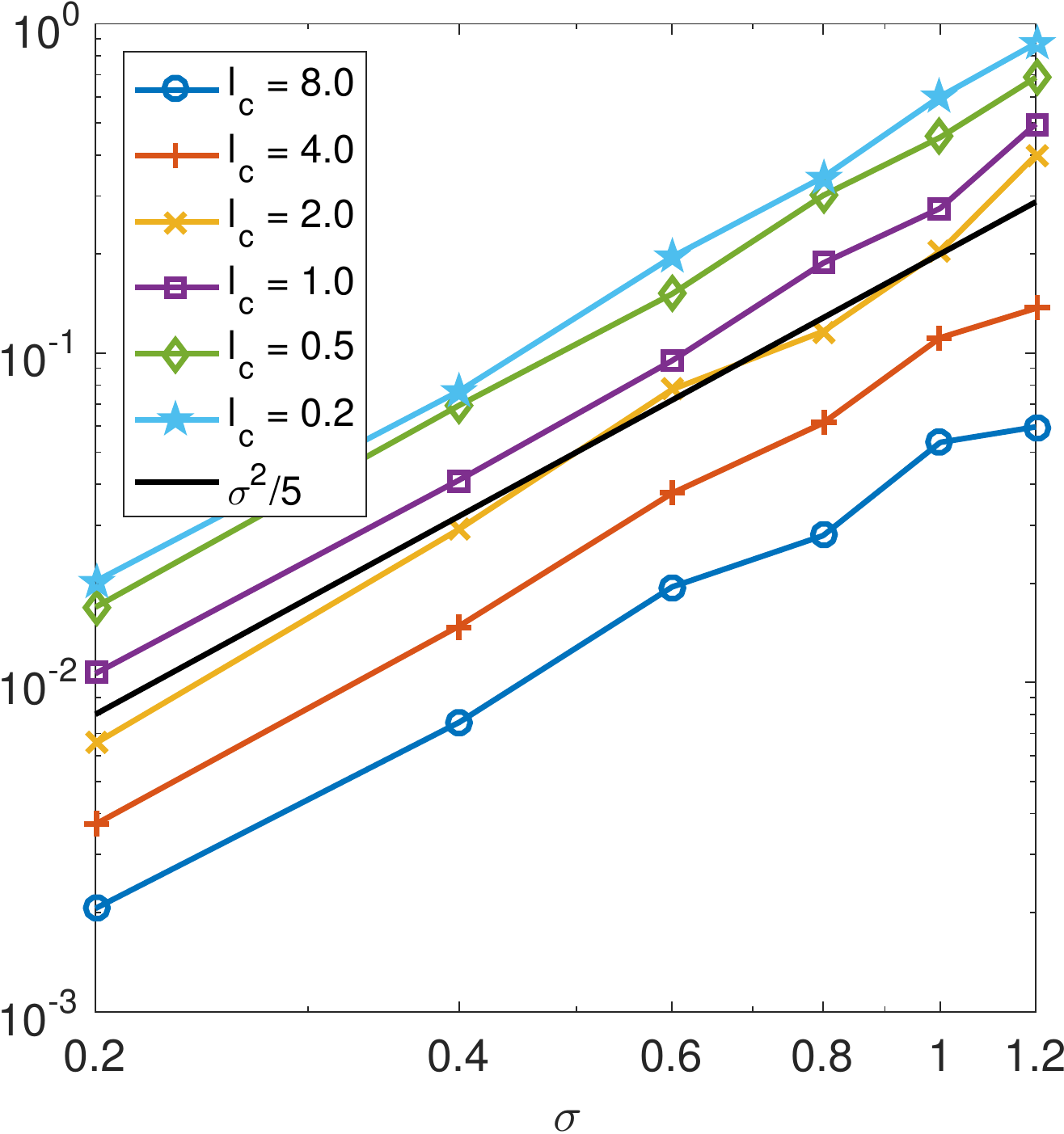}}\par
    \end{centering}
	\caption{\label{fig:IS_exp-2d_VR} \Yu{The variance
        reduction for the 2-dimensional case with
        exponential kernel having different correlation
        lengths and different values of $\sigma$. The $y$-axes are
        $\|\var(\tilde{u}_{I,h})\|_{H_0^1(D)}/\|\var(u_{I,h})\|_{H_0^1(D)}$.
        $N_{mc}=1000$ samples are used for the Monte Carlo
        method.  The tolerance of K-L expansion is set to
        $3\times 10^{-2}$.  The values of $M, p$
        corresponding to the stochastic Galerkin
        approximation of Model II with
        $l_c=8, 4, 2, 1, 0.5, 0.25$ are
        $(5,5), (11,4), (19,3), (28,3), (35, 2), (40,2)$, respectively.
        Note that $\log$ scales are used for both $x$ and
        $y$ axes.  } }
  \end{figure}

  The average and standard deviations of Monte Carlo
  solutions at line $y=0$ for model I with and without using
  model II as a control variate are given in
  \Yu{Fig. \ref{fig:IS_exp-1d} (exponential kernel in 1-d),
    and Fig \ref{fig:IS_exp-2d} (exponential kernel in
    2-d). The results for Gaussian kernel are similar but
    easier to obtain.}  It is seen that variance reduction
  is achieved for \Yu{all $\sigma$, but for small $\sigma$
    value, the reduction is significant. To numerically
    verify how the variance reduction is related to $\sigma$
    and $l_c$, we solved the two models with different
    parameters: $l_c = 0.2, 0.4, 0.6, 0.8, 1.0, 1.2$ and
    $l_c=8, 4, 2, 1, 0.5, 0.25$. The corresponding results
    for 1-dimensional and 2-dimensional case with
    exponential kernel are given in
    Fig. \ref{fig:IS_exp-1d_VR} and \ref{fig:IS_exp-2d_VR}
    respectively. The standard deviation reduction
    \eqref{eqn:IS_err_estimate2} derived from Lemma
    \ref{lem:IS_err} is clearly verified.}

  \subsection{Using $A_{\sII}$ as a preconditioner}

  The results of using model II to precondition model I is
  given in Table
  \ref{tab:prec_normal-1d},\ref{tab:prec_exp-1d} (for 1-d
  cases) and Table
  \ref{tab:prec_normal-2d},\ref{tab:prec_exp-2d} (for 2-d
  cases). We set default relaxation parameter in Richardson
  iteration to $\gamma={1}/(1+\Yu{3}\sigma^{2})$.

  For \Yu{almost all} the cases, the preconditioned
  Richardson iteration and GMRES are both better than the
  commonly-used Gauss-Seidel iteration., {especially for
    large $l_c$ or small $\sigma$.} The iteration numbers of
  Richardson method and GMRES are much smaller than
  Gauss-Seidel method, meanwhile their increases with
  respect to the standard deviation parameter $\sigma$ are
  also slower\Yu{, except for the cases with $p=1$}.  For
  large variance, the {preconditioned} GMRES method behaves
  much better than Gauss-Seidel and Richardson methods.
  \Yu{Note that we use the solution of model II as initial values for Richardson and GMRES iterations,
  so in the cases that model II is a very good approximation of model I, the corresponding iteration numbers
  are 0.}
  
  {According to our understanding of $u_{\sII}$, the worst
    scenario for the proposed preconditioners is when $l_c$
    is small and $\sigma$ is large. }  In a very few cases
  (e.g. \Yu{$l_c=0.2$ and $\sigma=0.6, 1$ in Table
    \ref{tab:prec_normal-2d}, \ref{tab:prec_exp-2d}}), the
  preconditioned Richardson iteration requires more
  iterations to converge than Gauss-Seidel, this probably
  because a first order Wiener Chaos approximation is used,
  the big approximation error together with the big modeling
  error deteriorate the performance of the preconditioning
  and the parameter $\omega$ in the Richardson method is not
  optimal.

  Based on the above observations, we advocate to use GMRES
  with model II as a preconditioner for solving the model I.
  
  \Yu{In the end, we compare our approach with some existing
    methods by solving a test problem studied in
    \cite{PU_SIAMJMAA2010}. The physical domain is set to
    $[0,1]^2$, {and} the force term $f(\bx)=1$. {The underlying Gaussian field of the log-normal coefficient $a(\bx, \omega)$} has a correlation function
    $K(\bx_1,\bx_2)=\sigma^2 r K_1(r)$, where
    $r=\| \bx_1-\bx_2 \|_2$ and $K_1$ is the modified Bessel
    function of the second kind with order one.  Set $M=5$
    in the K-L expansion, such that $97\%$ of the Gaussian
    field's total variance {is} captured.  The iteration
    numbers of Richardson and GMRES method for the
    stochastic Galerkin method of model I with model II as
    preconditioner for different $\sigma$ and $p$ are given
    in Table \ref{tab:prec_Powell2d}. From the table, we see
    that both Richardson and GMRES methods are efficient.
    As $p$ increases, the iteration numbers increase
    slowly. As $\sigma$ increases, the iteration numbers
    also increase slowly. The preconditioning effects are
    still very good for the cases with $\sigma=1$.  These
    results are very competitive comparing to the algebraic
    preconditioners studied in \cite{PU_SIAMJMAA2010} for
    this test example. }

  \begin{table}
	\begin{centering}
      \begin{tabular}{|c|c|c|c|c|c|c|c|}
        \hline 
        $l_{c}$ & $\sigma$ & $M$ & $p$ & $N_{M,p}$
        & $n_\text{GS}$ & $n_{\gamma}$ & $n_\text{GMRES}$\\
        \hline 
        20 & 0.2 & 1 & 10 & 11 & 3 & 0 & 0\\
        \hline 
        20 & 0.6 & 1 & 10 & 11 & 27 & 0 & 0\\
        \hline 
        20 & 1 & 1 & 10 & 11 & $>$100 & 22 & 5\\
        \hline 
        2 & 0.2 & 3 & 10 & 286 & 3 & 1 & 1\\
        \hline 
        2 & 0.6 & 3 & 10 & 286 & 22 & 3 & 1\\
        \hline 
        2 & 1 & 3 & 10 & 286 & $>$100 & 19 & 9\\
        \hline 
        0.2 & 0.2 & 11 & 3 & 364 & 3 & 1 & 1\\
        \hline 
        0.2 & 0.6 & 11 & 3 & 364 & 10 & 5 & 5\\
        \hline 
        0.2 & 1 & 11 & 3 & 364 & 29 & 12 & 9\\
        \hline 
      \end{tabular}\par
         \end{centering}
         \caption{\label{tab:prec_normal-1d}Preconditioning
           results of 1-dimensional problem with Gaussian
           kernel.  \Yu{$n_\text{GS}$, $n_{\gamma}$,
             $n_\text{GMRES}$ means the iteration number of
             Gauss-Seidel, Richardson and GMRES,
             respectively.  We take $\gamma=1/(1+3\sigma^2)$
             for the Richardson method.  The tolerance of
             K-L expansion is set to $2\times 10^{-3}$. The
             relative tolerance for the iteration solvers
             {is} set to $10^{-3}$.  } }
  \end{table}

  \begin{table}
	\centering{}%
	\begin{tabular}{|c|c|c|c|c|c|c|c|}
      \hline 
      $l_{c}$ & $\sigma$ & $M$ & $p$ & $N_{M,p}$
 & $n_\text{GS}$ & $n_{\gamma}$ & $n_\text{GMRES}$\\
      \hline 
      20 & 0.2 & 2 & 10 & 66 & 3 & 0 & 0\\
      \hline 
      20 & 0.6 & 2 & 10 & 66 & 24 & 2 & 1\\
      \hline 
      20 & 1 & 2 & 10 & 66 & $>100$ & 16 & 9\\
      \hline 
      2 & 0.2 & 8 & 5 & 1287 & 3 & 1 & 1\\
      \hline 
      2 & 0.6 & 8 & 5 & 1287 & 17 & 4 & 3\\
      \hline 
      2 & 1 & 8 & 5 & 1287 & $>100$ & 9 & 9\\
      \hline 
      0.2 & 0.2 & 51 & 2 & 1378 & 3 & 1 & 1\\
      \hline 
      0.2 & 0.6 & 51 & 2 & 1378 & 7 & 5 & 3\\
      \hline 
      0.2 & 1 & 51 & 2 & 1378 & 15 & 7 & 6\\
      \hline 
	\end{tabular} 
    \caption{\label{tab:prec_exp-1d}Preconditioning results
      of 1-dimensional problem with exponential kernel.
      \Yu{$n_\text{GS}$, $n_{\gamma}$, $n_\text{GMRES}$
        means the iteration number of Gauss-Seidel,
        Richardson and GMRES, respectively.  We take
        $\gamma=1/(1+3\sigma^2)$ for the Richardson method.
        The tolerance of K-L expansion is set to
        $3\times 10^{-2}$. The relative tolerance for the
        iteration solvers {is} set to $10^{-3}$.}  }
  \end{table}

  \begin{table}
	\centering{}%
	\begin{tabular}{|c|c|c|c|c|c|c|c|}
      \hline 
      $l_{c}$ & $\sigma$ & $M$ & $p$ & $N_{M,p}$
              & $n_\text{GS}$ & $n_{\gamma}$ & $n_\text{GMRES}$\\
      \hline 
      20 & 0.2 & 1 & 16 & 17 & 3 & 0 & 0\\
      \hline 
      20 & 0.6 & 1 & 16 & 17 & 25 & 0 & 0\\
      \hline 
      20 & 1 & 1 & 16 & 17 & 29 & 1 & 1\\
      \hline 
      2 & 0.2 & 4 & 5 & 126 & 3 & 0 & 0\\
      \hline 
      2 & 0.6 & 4 & 5 & 126 & 17 & 5 & 4\\
      \hline 
      2 & 1 & 4 & 5 & 126 & 48 & 14 & 7\\
      \hline 
      0.2 & 0.2 & 80 & 1 & 81 & 2 & 1 & 1\\
      \hline 
      0.2 & 0.6 & 80 & 1 & 81 & 3 & 4 & 2\\
      \hline 
      0.2 & 1 & 80 & 1 & 81 & 4 & 7 & 3\\
      \hline 
	\end{tabular}
    \caption{\label{tab:prec_normal-2d}Preconditioning
      results of 2-dimensional problem with Gaussian kernel.
      \Yu{$n_\text{GS}$, $n_{\gamma}$, $n_\text{GMRES}$
        means the iteration number of Gauss-Seidel,
        Richardson and GMRES, respectively.  We take
        $\gamma=1/(1+3\sigma^2)$ for the Richardson method.
        The tolerance of K-L expansion is set to
        $10^{-2}$. The relative tolerance for the iteration
        solvers {is} set to $10^{-3}$.}  }
  \end{table}

  \begin{table}
	\centering{}%
	\begin{tabular}{|c|c|c|c|c|c|c|c|}
      \hline 
      $l_{c}$ & $\sigma$ & $M$ & $p$ & $N_{M,p}$
 	& $n_\text{GS}$ & $n_{\gamma}$ & $n_\text{GMRES}$\\
      \hline 
      20 & 0.2 & 3 & 8 & 165 & 3 & 0 & 0\\
      \hline 
      20 & 0.6 & 3 & 8 & 165 & 12 & 1 & 1\\
      \hline 
      20 & 1 & 3 & 8 & 165 & 41 & 14 & 10\\
      \hline 
      2 & 0.2 & 28 & 2 & 435 & 3 & 1 & 1\\
      \hline 
      2 & 0.6 & 28 & 2 & 435 & 4 & 3 & 3\\
      \hline 
      2 & 1 & 28 & 2 & 435 & 10 & 9 & 4\\
      \hline 
      0.2 & 0.2 & 86 & 1 & 87 & 2 & 1 & 1\\
      \hline 
      0.2 & 0.6 & 86 & 1 & 87 & 2 & 3 & 2\\
      \hline 
      0.2 & 1 & 86 & 1 & 87 & 4 & 7 & 3\\
      \hline 
	\end{tabular}
    \caption{\label{tab:prec_exp-2d}Preconditioning results
      of 2-dimensional problem with exponential kernel.
      \Yu{$n_\text{GS}$, $n_{\gamma}$, $n_\text{GMRES}$
        means the iteration number of Gauss-Seidel,
        Richardson and GMRES, respectively.  We take
        $\gamma=1/(1+3\sigma^2)$ for the Richardson method.
        The tolerance of K-L expansion is set to
        $3\times10^{-2}$. The relative tolerance for the
        iteration solvers {is} set to $10^{-3}$.  } }
  \end{table}

\begin{table}
	\centering{}%
	\begin{tabular}{|c|c|c|c|c|c||c|c|c|c|c|}
		\hline 
		   & \multicolumn{5}{|c||}{Richardson} & \multicolumn{5}{c|}{GMRES} \\
		\hline
		$\sigma$ & $p=1$ & $p=2$ & $p=3$ & $p=4$ & $p=5$ 
		& $p=1$ & $p=2$ & $p=3$ & $p=4$ & $p=5$ \\
		\hline 
		0.2 & 5 & 6 & 5 & 6 & 6 & 3 & 3 & 4 & 4 & 4 \\
		\hline 
		0.4 & 10 & 10 & 11 & 10 & 10 & 3 & 4 & 5 & 6 & 7 \\
		\hline 
		0.6 & 14 & 16 & 17 & 18 & 19 & 4 & 5 & 6 & 7 &  8\\
		\hline 
		0.8 & 16 & 19 & 21 & 23 & 25 & 5 & 6 & 7 & 8 & 9 \\
		\hline 
		1.0 & 16 & 19 & 21 & 24 & 26 & 5 & 7 & 8 & 9 &  11\\
		\hline 
	\end{tabular}
	\caption{\label{tab:prec_Powell2d}\Yu{The iteration
        numbers of Richardson and GMRES method solving the
        2-dimensional problem with Matern-tye kernel studied
        in \cite{PU_SIAMJMAA2010}. We take
        $\gamma=1/(1+3\sigma^2)$ for the Richardson method.
        The relative tolerance for the iteration solvers {is} 
        set to $10^{-8}$.  $M=5$.  } }
\end{table}

\section{Summary}

In this work, we consider the Wick approximation of two
stochastic elliptic problems with log-normal random
coefficients, where Model II is a second order approximation
of model I with respect to $\sigma$. Model II can be used as
a precondition for model I in a stochastic Galerkin
method. The numerical results show that the preconditioned
Richardson iteration is better than commonly used
Gauss-Seidel method when $\sigma$ is small or $l_c$ is
large. Meanwhile, the former method have a parameter to
tune. The preconditioned GMRES method works very well for
all the values of $\sigma$ and $l_c$ tested using defaults
parameters. The model II can also be used as an efficient
important sampling process for model I to reduce the
variance of a Monte Carlo approach when the stochastic
dimension in a Karhunen-Lo\`{e}ve expansion is very high.

\acknowledgements

The work of X. Wan was partially supported by a NSF grant
DMS-1620026.  The work of H. Yu was partially supported by
China National Program on Key Basic Research Project
2015CB856003, NNSFC Grant 11771439 and China Science
Challenge Project TZ2018001.

\appendix

\end{document}